\documentclass{amsart}
\usepackage{graphics}
\usepackage[all, cmtip]{xy}
\usepackage{mathabx}
\begin{document}

\newtheorem{theorem}{Theorem}
\newtheorem{conj}{Conjecture}
\newtheorem{prop}{Proposition}
\newtheorem*{aptheorem}{``Theorem''}
\theoremstyle{definition}
\newtheorem{definition}{Definition}

\newcommand{\RR}{\mathbb{R}}
\newcommand{\CC}{\mathbb{C}}
\newcommand{\NN}{\mathbb{N}}
\newcommand{\ZZ}{\mathbb{Z}}
\newcommand{\QQ}{\mathbb{Q}}
\newcommand{\ol}{\overline}
\newcommand{\cA}{\mathcal{A}}
\newcommand{\cD}{\mathcal{D}}
\newcommand{\cF}{\mathcal{F}}
\newcommand{\cM}{\mathcal{M}}
\newcommand{\tcM}{\widetilde{\mathcal{M}}}
\newcommand{\wt}{\widetilde}
\newcommand{\alphas}{\boldsymbol{\alpha}}
\newcommand{\betas}{\boldsymbol{\beta}}
\renewcommand{\a}{\alpha}
\renewcommand{\b}{\beta}
\newcommand{\g}{\gamma}
\renewcommand{\S}{\Sigma}
\renewcommand{\H}{\mathcal{H}}
\newcommand{\x}{\mathbf{x}}
\newcommand{\y}{\mathbf{y}}
\newcommand{\T}{\mathbb{T}}
\newcommand{\s}{\mathfrak{s}}
\newcommand{\Spinc}{\text{Spin}^c}
\newcommand{\gr}{\text{gr}}
\newcommand{\M}{\mathcal{M}}
\newcommand{\wM}{\widetilde{\M}}
\newcommand{\HFh}{\widehat{HF}}
\newcommand{\red}{\text{red}}
\newcommand{\HFK}{\widehat{HFK}}
\newcommand{\HFL}{\widehat{HFL}}
\newcommand{\spinc}{\text{Spin}^c}
\newcommand{\sym}{\text{Sym}}
\renewcommand{\j}{\mathfrak{j}}
\newcommand{\z}{\underline{z}}

\title{A survey of Heegaard Floer homology}
\author{Andr\'as Juh\'asz}
\thanks{AJ was supported by a Royal Society Research Fellowship and OTKA grant NK81203}
\address{Mathematical Institute, University of Oxford, Andrew Wiles Building, Woodstock Road, Oxford, OX2 6GG, UK}
\email{juhasza@maths.ox.ac.uk}
\maketitle

Since its inception in 2001, Heegaard Floer homology has developed
into such a large area of low-dimensional topology that it has become
impossible to overview all of its applications and ramifications in a
single paper. For the state of affairs in 2004,
see the excellent survey article of Ozsv\'ath and Szab\'o~\cite{OSz-survey}.
A decade later, this work has two goals.  The first is to provide a conceptual
introduction to the theory for graduate students and interested
researchers, the second is to survey the current state of the
field, without aiming for completeness.

After reviewing the structure of Heegaard Floer homology, treating it as a black box,
we list some of its most important applications. Many of these
are purely topological results, not referring to Heegaard Floer homology itself.
Then, we briefly outline the construction of Lagrangian intersection Floer homology,
as defined by Fukaya, Oh, Ono, and Ohta~\cite{FOOO}.
Given a strongly $\s$-admissible based Heegaard diagram $(\S,\alphas,\betas,z)$ of the $\spinc$ $3$-manifold~$(Y,\s)$,
we construct the Heegaard Floer chain complex $CF^\infty(\S,\alphas,\betas,z,\s)$ as a special case of the above,
and try to motivate the role of the various seemingly ad hoc features such as admissibility,
the choice of basepoint, and $\spinc$-structures. We also discuss the
proof of invariance of the homology $HF^\infty(\S,\alphas,\betas,\s)$
up to isomorphism under all the choices made, and how to define
$HF^\infty(Y,\s)$ using this in a functorial way (naturality).
Next, we explain why Heegaard Floer
homology is computable, and how  it lends itself to the various
combinatorial descriptions that most students encounter first during their
studies. The last chapter gives an overview of the definition
and applications of sutured Floer homology, which includes sketches of
some of the key proofs.

Throughout, we have tried to collect some of the important open conjectures in the area.
For example, a positive answer to two of these would give a new proof of
the Poincar\'e conjecture.

\subsection*{Acknowledgement} I would like to thank Lino Campos Amorim, Dominic Joyce, Yanki Lekili,
Ciprian Manolescu, Peter Ozsv\'ath, Alexander Ritter, and Zolt\'an Szab\'o
for helpful discussions, and Fyodor Gainullin, Cagatay Kutluhan, Marco Marengon, Goncalo Oliveira,
Jacob Rasmussen, and Andr\'as Stipsicz for their comments on earlier versions of this paper.

\section{Background}

4-manifold topology was revolutionized in 1982 by the work of
Donaldson, who pioneered techniques coming from theoretical physics,
namely gauge theory, to study smooth 4-manifolds. In order to obtain
an interesting invariant of the smooth structure, the idea is to fix
some additional geometric structure on the manifold such as a
Riemannian metric, write down a non-linear PDE, and then study the
topology of the moduli space of solutions. If one is lucky enough,
this does not depend on the additional choices made, only on the
smooth structure.  The Donaldson polynomial
invariants~\cite{Donaldson_poly} arise from the cohomology of a
certain compactification of the moduli space of $SU(2)$ Yang-Mills
instantons over a Riemannian 4-manifold. They are independent of the
choice of metric, but do depend on the smooth
structure. Unfortunately, the lack of compactness often makes it
difficult to work with.

Based on arguments coming from string theory, Seiberg and Witten wrote
down a different set of equations whose solution spaces are usually
compact, and hence easier to work with. They also wrote down a
conjectural relationship between their theory and the Donaldson
polynomial invariants.

Heegaard Floer homology was defined by Ozsv\'ath and
Szab\'o~\cite{OSz1, OSz8}.  It grew out of an attempt to better understand
the Seiberg-Witten invariant of closed 4-manifolds, so that it
becomes more computable.  It consists of a package of invariants of
closed oriented 3-manifolds, maps induced on these by cobordisms, and
a 4-manifold invariant obtained via mixing the various flavors.
Simultaneously, an analogous theory was developed by Kronheimer and
Mrowka, called monopole Floer homology, based directly on the
Seiberg-Witten monopole equations. The two theories have been shown to
be equivalent.  The proof passes through a third type of invariant of
3-manifolds called embedded contact homology (ECH), due to Hutchings.
This is defined in terms of a contact structure on the
3-manifold, but turns out to be independent of this choice.  The
Heegaard Floer, monopole, and ECH invariants of 3-manifolds are
equivalent, but they are adapted to different aspects of 3-manifold
and contact topology.

The motivation for the definition of $HF$ was provided by Atiyah's~\cite{Atiyah} topological quantum field
theory (TQFT) picture and the Atiyah-Floer conjecture~\cite{AF}.
A $(3+1)$-dimensional TQFT over $\ZZ$ assigns to a closed oriented 3-manifold~$Y$ a
finitely generated Abelian group~$Z(Y)$,
and to an oriented smooth cobordism~$W$ from~$Y$ to~$Y'$ a
homomorphism $F_W \colon Z(Y) \to Z(Y')$. This assignment satisfies certain axioms.
For example, it is functorial from the cobordism category of oriented $3$-manifolds to the category of
finitely generated Abelian groups, and $Z(\emptyset) = \ZZ$.
Given a smooth oriented 4-manifold~$X$, one can view it as a cobordism from~$\emptyset$ to~$\emptyset$. Then
$F_X \colon \ZZ  \to  \ZZ$
is multiplication by some integer $n(X)$, which is an invariant of the
smooth 4-manifold $X$.

The Seiberg-Witten invariant of a smooth $\Spinc$ 4-manifold $(X,\s)$ with $b_2^+(X) \ge 2$
is an integer $SW(X,\s)$ obtained by choosing a Riemannian metric~$g$
on $X$, and considering the moduli space of solutions to the so called monopole equations,
up to gauge equivalence. This moduli space is a compact oriented manifold of dimension
\[
d(\s) = \frac{c_1(\s)^2 - 2\chi(X) - 3\sigma(X)}{4}
\]
for a generic~$g$. The moduli spaces corresponding to different generic metrics are cobordant.
Hence, when $d(\s) = 0$, the signed count of elements in this
0-dimensional moduli space is independent of the choice of~$g$,
giving rise to the invariant $SW(X,\s)$.

Unfortunately, the Seiberg-Witten invariant does not quite fit in the above TQFT picture,
but similar constructions are available due to Kronheimer and Mrowka (monopole Floer homology, $HM$)
and to Ozsv\'ath and Szab\'o ($HF$). For several years, it was conjectured that
these two theories are isomorphic, and this has recently been settled by the work of several people.
Surprisingly, the equivalence between $HM$ and $HF$ passes through embedded contact homology ($ECH$),
defined by Hutchings~\cite{ECH1} and Hutchings-Taubes~\cite{ECH2, ECH3}.
To define $ECH$, one starts out with a contact 3-manifold~$(Y,\a)$ (i.e., $\a$ is a $1$-form such that $\a \wedge d\a > 0$),
and the chain complex is generated by certain periodic Reeb orbits
lying in a given homology class $\Gamma \in H_1(Y)$. Taubes~\cite{ECH=SW} proved that
\[
\widecheck{HM}(Y,\s_\a + PD[\Gamma]) \cong ECH(-Y,\a, \Gamma),
\]
where $\s_\a$ is the $\Spinc$-structure given by $\ker(\a)$.
This establishes that $ECH(Y,\a)$ only depends on the $3$-manifold~$Y$.
As of now, there is no intrinsic proof of this fact,
and even showing that $ECH(Y,\a)$ only depends on the contact $2$-plane field $\ker(\a)$ is a formidable task.
Recently, Kutluhan-Lee-Taubes~\cite{HF=HM1, HF=HM2, HF=HM3, HF=HM4} proved that
\[
\widecheck{HM}(Y,\s) \cong HF^+(Y,\s),
\]
passing through a version of embedded contact homology for stable Hamiltonian structures
generalizing contact structures.
Meanwhile, Colin-Ghiggini-Honda~\cite{CGH, CGH2, CGH3} showed that
\[
ECH(-Y,\a,\Gamma) \cong HF^+(Y,\s_\a + PD(\Gamma)),
\]
which, together with the isomorphism constructed by Taubes, also gives
\[
\widecheck{HM}(Y,\s) \cong HF^+(Y,\s).
\]

\section{Overview of the structure of $HF$}

Given a closed, connected, oriented 3-manifold $Y$ and a $\Spinc$-structure $\s \in \Spinc(Y)$,
Heegaard Floer homology assigns to the pair $(Y,\s)$ a finitely generated Abelian group
$\HFh(Y,\s)$, and $\ZZ[U]$-modules $HF^\infty(Y,\s)$, $HF^+(Y,\s)$, and $HF^-(Y,\s)$,
where $\ZZ[U]$ is the polynomial ring in the formal variable $U$. In the future, when
we write $HF^\circ$, we mean one of these four flavors of Heegaard Floer homology.
Furthermore, let
\[
HF^\circ(Y) = \bigoplus_{\s \in \Spinc(Y)} HF^\circ(Y,\s).
\]

Each of these groups carries a relative $\ZZ_{d(c_1(\s))}$-grading, where $d(c_1(\s))$
is the divisibility of the cohomology class $c_1(\s) \in H^2(Y)$.
(If $A$ is a finitely generated Abelian group, the divisibility $d(a)$ of $a \in A$ is $0$ if $a$ is torsion,
and otherwise the image of $a$ in $A/ \text{Tors}$ is $d(a)$ times a primitive element.)
When $c_1(\s)$ is torsion, then $d(c_1(\s)) = 0$, and the relative $\ZZ$-grading can be lifted to an absolute $\QQ$-grading.
Furthermore, each group $HF^\circ(Y)$ carries an absolute $\ZZ_2$-grading.
If $b_1(Y) > 0$ and $\s$ is a non-torsion $\spinc$-structure,
and one takes the Euler characteristic of $HF^+(Y,\s)$ with respect to this $\ZZ_2$-grading,
then one recovers the Turaev torsion of~$Y$ in the $\Spinc$-structure $\s$, cf.\ Turaev~\cite{torsion}
(when $b_1(Y) = 1$, the torsion is calculated in the ``chamber'' containing $c_1(\s)$).
On the other hand,
\[
\chi \left(\HFh(Y,\s) \right) =
\left\{
\begin{array}{l l}
1 & \quad \text{if $b_1(Y) = 0$} \\
0 & \quad \text{if $b_1(Y) > 0$}
\end{array}
\right.
\]
for every~$\s \in \spinc(Y)$.

The three flavors of Heegaard Floer homology are related by the exact sequence
\begin{equation} \label{eqn:LES}
\dots \longrightarrow HF^-(Y,\s) \stackrel{\iota}{\longrightarrow}  HF^\infty(Y,\s) \stackrel{\pi}{\longrightarrow}
 HF^+(Y,\s) \stackrel{\delta}{\longrightarrow} \dots
\end{equation}
This gives rise to the invariant
\[
HF^+_\red(Y,\s) = \text{coker}(\pi) \cong \ker(\iota) = HF^-_\red(Y,\s),
\]
where the isomorphism is given by the coboundary map $\delta$.

Furthermore, we have an exact sequence
\begin{equation} \label{eqn:LES2}
\dots \longrightarrow \HFh(Y,\s) \stackrel{i}{\longrightarrow} HF^+(Y,\s) \stackrel{U^+}{\longrightarrow} HF^+(Y,\s)
\stackrel{p}{\longrightarrow} \dots,
\end{equation}
and similarly,
\[
\dots \longrightarrow HF^-(Y,\s) \stackrel{U^-}{\longrightarrow} HF^-(Y,\s) \longrightarrow \HFh(Y,\s)  \longrightarrow \dots
\]
There is one more piece of algebraic structure on $HF^\circ(Y,\s)$, which is an action of
the group $\Lambda^*\left( H_1(Y) / \text{Tors} \right)$.

When~$Y$ is a rational homology sphere, $HF^+(Y,\s)$ is absolutely $\QQ$-graded
for every $\s \in \spinc(Y)$. The \emph{correction term} $d(Y,\s)$, introduced by
Ozsv\'ath and Szab\'o~\cite{OSz15}, is the minimal grading of any non-torsion
element in the image of $HF^\infty(Y,\s)$ in $HF^+(Y,\s)$. This probably coincides with
the gauge-theoretic invariant of Fr{\o}yshov~\cite{Froyshov}.

Ozsv\'ath and Szab\'o~\cite{OSz1} showed that $HF^\circ$ is well-defined up to isomorphism,
and checked some naturality properties in~\cite{OSz10}. The assignment
$Y \mapsto HF^\circ(Y)$ was made completely functorial by Thurston and the author~\cite{naturality},
where we also showed that the mapping class group of $Y$ acts on $HF^\circ(Y)$.
Naturality is necessary to be able to talk about maps between $HF$ groups, and to
be able to talk about concrete elements. It turns out that the $+$, $-$, and
$\infty$ versions are indeed natural (in analogy with the corresponding
flavors of monopole Floer homology, where there is no basepoint dependence).
However, $\HFh$ is only functorial
on the category of based $3$-manifolds and basepoint preserving diffeomorphisms
(this is work in progress joint with Ozsv\'ath and Thurston).
Indeed, let~$\g$ be a loop in~$Y$ passing through the basepoint~$p$, and consider the automorphism
$d$ of $(Y,p)$ which is a finger move along~$\g$. For $x \in \HFh(Y,p)$, we have
\[
d_*(x) = x + p \circ i ([\g] \cdot x),
\]
where $p$ and $i$ are the maps in the
exact sequence~\eqref{eqn:LES2}, while~$[\g]$ is the class of the curve~$\g$
in $H_1(Y)/\text{Tors}$. This map is non-trivial for example
when
\[
Y = \S(2,3,7) \# (S^1 \times S^2)
\]
and $\g = S^1 \times \{\text{pt}\}$, the basepoint being an arbitrary element of~$\g$.

Heegaard Floer homology enjoys various symmetry properties.
There is an involution on the set $\spinc(Y)$, denoted by $\s \mapsto \ol{\s}$.
If~$\s$ is represented by a vector field~$v$, then $\ol{\s}$ is represented by~$-v$.
Ozsv\'ath and Szab\'o~\cite{OSz8} showed that
\[
HF^\circ(Y,\s) \cong HF^\circ(Y,\ol{\s})
\]
as $\ZZ[U] \otimes_{\ZZ} \Lambda^*(H_1(Y)/\text{Tors})$-modules.
The Heegaard Floer chain complexes gives rise to both homology and cohomology
theories. We denote by $\HFh_*$, $HF^+_*$, and $HF^-_*$ the homologies,
and $\HFh^*$, $HF^*_+$, and $HF^*_-$ the corresponding cohomologies, respectively.
If~$-Y$ denotes~$Y$ with its orientation reversed, then
\[
\HFh^*(Y,\s) \cong \HFh_*(-Y,\s) \text{ and } HF^*_\pm(Y,\s) \cong HF^\mp_*(-Y,\s).
\]

As in a TQFT, cobordisms of 3-manifolds induce homomorphisms. More precisely,
if $(W,\s)$ is a $\Spinc$-cobordism from $(Y_0,\s_0)$ to $(Y_1,\s_1)$,
then Ozsv\'ath and Szab\'o~\cite{OSz10} associate to it a map
\[
F_{W,\s}^\circ \colon HF^\circ(Y_0,s_0) \to HF^\circ(Y_1,\s_1).
\]
When $\s_1$ and $\s_2$ are both torsion, this homomorphism shifts the absolute $\QQ$-grading by the number
\[
d(\s) = \frac{c_1(\s)^2 - 2\chi(W) - 3\sigma(W)}{4}
\]
(note the ``coincidence'' with the dimension of the Seiberg-Witten moduli space).
More generally, there is also a map
\[
F^\circ_{W,\s} \colon HF^\circ(Y_0,\s_0) \otimes \Lambda^* \left( H_1(W)/\text{Tors} \right) \to HF^\circ(Y_1,\s_1).
\]
When $b_2^+(W) > 0$, then $F^\infty_{W,\s} = 0$.

Let $W \colon Y_0 \to Y_1$ be a cobordism with $b_2^+(W) \ge 2$.
An \emph{admissible cut} of the cobordisms $W$ is a 3-manifold $N \subset W$ such that
\begin{itemize}
\item $N$ divides $W$ into two cobordisms $W_1 \colon Y_0 \to N$ and $W_2 \colon N \to Y_1$,
\item $b_2^+(W_1) \ge 1$ and $b_2^+(W_2) \ge 1$,
\item $\delta H^1(N) = 0$ in $H^2(W,\partial W)$.
\end{itemize}
An admissible cut always exists.
Since $F^\infty_{W_1,\s} = 0$, the long exact sequence~\eqref{eqn:LES} implies that
the image of the map
\[
F^-_{W_1,\s|_{W_1}} \colon HF^-(Y_0,\s|_{Y_0}) \to HF^-(N,\s|_N)
\]
lies in $HF^-_\red(N,\s|_N)$. Similarly, the map
\[
F^+_{W,\s|_{W_2}} \colon HF^+(N,\s|_N) \to HF^+(Y_1,\s|_{Y_1})
\]
factors through the projection of $HF^+(N,\s|_N)$ to $HF^+_\red(N,\s|_N)$.
We define the \emph{mixed invariant}
\[
F_{W,\s}^\text{mix}
\colon HF^-(Y_0,\s|_{Y_0}) \otimes_\ZZ \Lambda^*(H_1(W)/\text{Tors}) \to HF^+(Y_1,\s|_{Y_1})
\]
using the formula
\[
F_{W,\s}^\text{mix} = F^+_{W_1,\s|_{W_1}} \circ \tau^{-1} \circ F^-_{W_0,\s|_{W_0}},
\]
where $\tau \colon HF^-_\red(N,\s|_N) \to HF^+_\red(N,\s|_N)$ is the isomorphism
induced by the co\-boundary map $\delta$ in the long exact sequence~\eqref{eqn:LES}.
As the notation suggests, the map~$F_{W,\s}^\text{mix}$ is independent of the choice
of admissible cut $N$.

Given a closed oriented smooth 4-manifold $X$ with $b_2^+(W) \ge 2$ and a $\Spinc$-structure $\s \in \Spinc(X)$,
we define the \emph{absolute invariant} $\Phi_{X,\s}$ as follows.
First, let $W = X \setminus 2B^4$, this can be viewed as
a cobordism from $S^3$ to $S^3$. We will also write~$\s$ for the restriction of~$\s$ to~$W$.
Note that there is a unique $\Spinc$-structure~$\s_0$ on~$S^3$, and that
$HF^-(S^3,\s_0) \cong \ZZ[U]$, while
\[
HF^+(S^3,\s_0) \cong \ZZ[U, U^{-1}] / U \ZZ[U].
\]
We write~$\Theta_-$ for a
generator of the 0-degree part of $HF^-(S^3,\s_0)$ and~$\Theta_+$ for a generator of the 0-degree part of $HF^+(S^3,\s_0)$,
these are both well-defined up to sign. Then the map
\[
\Phi_{X,\s} \colon \ZZ[U] \otimes \Lambda^*(H_1(X)/\text{Tors}) \to \ZZ/\pm1
\]
is defined by taking $\Phi_{X,\s}(U^n \otimes \zeta)$ to be the coefficient of $\Theta_+$ in
$F^\text{mix}_{W,\s}(U^n \cdot \Theta_- \otimes \zeta)$.
Note that $\Phi_{X,\s}$ vanishes on those homogeneous elements whose degree is different from $d(\s)$.
The sign ambiguity comes from the choice of $\Theta_\pm$.
Ozsv\'ath and Szab\'o conjectured that one can recover the Seiberg-Witten invariants from this as follows.

\begin{conj}
Take a basis $b_1,\dots,b_k$
of $H_1(X)/\text{Tors}$, and let $n$ be such that the degree of $U^n \otimes (b_1 \wedge \dots \wedge b_k)$ is $d(\s)$.
Then
\[
SW(X,\s) = \Phi_{X,\s}(U^n \otimes (b_1 \wedge \dots \wedge b_k)).
\]
\end{conj}

Note that there is a more general version of $SW$ that is obtained by integrating different
elements of the cohomology of the configuration space over the Seiberg-Witten moduli space,
that is conjectured to agree with $\Phi_{X,\s}$.

The Heegaard Floer package also contains a knot invariant, called \emph{knot Floer homology}.
Given a null-homologous knot or link $K$ in a closed, connected, oriented 3-manifold $Y$, Ozsv\'ath and Szab\'o~\cite{OSz3},
and independently Rasmussen~\cite{ras}, assigned to it
a finitely generated Abelian group $\HFK(Y,K)$. This refines $\HFh(Y)$ in the sense that
there is a filtration on the chain complex defining $\HFh(Y)$ such that the homology of the
associated graded object is $\HFK(Y,K)$. Consequently, there is a spectral sequence
from $\HFK(Y,K)$ converging to $\HFh(Y)$. When $Y = S^3$, the smallest filtration level
for which the inclusion map on homology is non-zero into $\HFh(S^3) \cong \ZZ$ is denoted
by~$\tau(K)$, see~\cite{OSz14}.

In the case of $Y = S^3$, the group $\HFK(K) = \HFK(S^3,K)$ is bi-graded; i.e.,
\[
\HFK(K) = \bigoplus_{i, j \in \ZZ} \HFK_j(K,i).
\]
Here $i$ is called the Alexander grading and $j$ is the homological grading.
This is justified by the fact that
\[
\sum_{i,j \in \ZZ} (-1)^j \cdot \text{rk}\left(\HFK_j(K,i) \right) t^i
\]
is the symmetrized Alexander polynomial $\Delta_K(t)$ of $K$.
The proof relies on the fact that knot Floer homology satisfies an unoriented skein exact sequence.
Another way of saying this is that knot Floer homology categorifies the Alexander polynomial, just
like Khovanov homology is a categorification of the Jones polynomial.
According to a conjecture of Rasmussen~\cite{superpoly, knot-homol}, the two theories are related.
\begin{conj} \label{conj:ss}
There is a spectral sequence starting from the reduced Khovanov
homology of $K$ and converging to $\HFK(K)$. In particular,
\[
\text{rk}\left(\widetilde{\text{Kh}}(K)\right) \ge \text{rk}\left(\HFK(K)\right).
\]
\end{conj}

The author showed in~\cite{cob} that knot cobordisms induce maps on knot Floer homology,
making the categorification complete.
However, to make knot Floer homology functorial, one needs to work with based knots, according to the work of Sarkar~\cite{basepoint}.

It is a classical result that the degree of $\Delta_K(t)$ provides a lower bound on the
Seifert genus $g(K)$ (which is the minimal genus of an oriented surface in $S^3$ bounded by $K$).
Ozsv\'ath and Szab\'o showed in~\cite{OSz6} that in fact knot Floer homology detects the Seifert genus in the sense that
\[
g(K) = \max \{\, i \in \ZZ \,\colon\, \HFK_*(K,i) \neq 0 \,\}.
\]
Furthermore, by work of Ghiggini-Ni~\cite{Ghiggini, fibred, corrigendum}
and the author~\cite{decomposition, polytope}, the knot~$K$ is fibred if and only if
\[
\HFK_*(K,g(K)) \cong \ZZ.
\]

A generalization of knot Floer homology, also due to Ozsv\'ath and Szab\'o~\cite{OSz2}, is called
\emph{link Floer homology}. Given a link $L$ in $S^3$, this invariant is denoted by $\HFL(L)$,
and is graded by $H_1(S^3 \setminus L)$. Its graded Euler characteristic gives rise to the
multivariable Alexander polynomial of $L$, and it detects the Thurston norm of the link complement.

Since knot Floer homology detects the genus, it is sensitive to Conway mutation. E.g., it distinguishes
the Conway and the Kinoshita-Teresaka knots, as the first one has genus~$3$, while the second one
has genus~$2$. The $\delta$-grading on $\HFK(K)$ is defined as the difference of the Alexander
and the homological gradings. Then we have the following conjecture, communicated to me by Zolt\'an
Szab\'o, and supported by computational evidence.

\begin{conj}
The rank of knot Floer homology is unchanged by Conway mutation in each $\delta$-grading.
\end{conj}

So far, $\HFL(L)$ has proved to been torsion-free in each example computed.

\begin{conj} \label{conj:free}
The group $\HFL(L)$ is torsion free for every link~$L$.
\end{conj}

Knot and link Floer homology are invariant of the link complement. It is natural to ask whether
they are particular cases of some more general invariant for 3-manifolds with boundary.
There are two existing such theories, namely, sutured Floer homology (SFH) due to the author~\cite{sutured},
and bordered Floer homology due to Lipshitz, Ozsv\'ath, and Thurston~\cite{bordered}.
However, for both, one needs more structure on the boundary. SFH is defined for sutured manifolds,
which were introduced by Gabai~\cite{Gabai, Gabai2}. A \emph{sutured manifold} is a pair $(M,\g)$, where $M$
is a compact oriented 3-manifold with boundary, and $\g \subset \partial M$ can be thought
of as a thickened oriented one-manifold that divides $\partial M$ into subsurfaces $R_-(\g)$
and~$R_+(\g)$. The components of~$\g$ are called the \emph{sutures}.
Then the sutured Floer homology $SFH(M,\g)$ is a finitely generated Abelian group.

If $p \in Y$ is a point, then the sutured manifold $Y(p)$ is obtained by removing a ball around
$p$ and putting a single suture on the boundary. For this, we have
\[
SFH(Y(p)) \cong \HFh(Y).
\]
Furthermore, if $L$ is a link in $Y$, then $Y(L)$ denotes the sutured manifold where $M = Y \setminus N(L)$,
and on each boundary torus, we have two oppositely oriented meridional sutures.
Then
\[
SFH(S^3(L)) \cong \HFL(L).
\]
So SFH is a common generalization of both $\HFh$ and $\HFL$.
Note that one can define $Y(p)$ and $Y(L)$ canonically using the real blow-up construction, cf.~\cite[Definitions~2.4 and~2.5]{naturality}.

The bordered Floer complex $\widehat{CFD}(M)$ is defined for a compact 3-manifold $M$ with connected parametrized boundary.
The parametrization amounts to fixing a handle decomposition of $\partial M$ with a single 0-handle.
Here, $\widehat{CFD}(M)$ is a differential graded $\mathcal{A}(\partial M)$-module, and $\mathcal{A}(\partial M)$
is a differential graded algebra that depends on the handle decomposition.
If $S$ is a surface in the 3-manifold $Y$ that cuts it into pieces $-M_1$ and $M_2$, then we have the gluing formula
\[
\HFh(Y) \cong \text{Mor}_{\mathcal{A}(S)}\left( \widehat{CFD}(M_1), \widehat{CFD}(M_2) \right).
\]
Bordered Floer homology is currently being developed at a rapid pace, and due to
space constraints, we refer the interested reader to the survey article
of Lipshitz, Ozsv\'ath, and Thurston~\cite{bordered-tour}.

It is worth pointing out the relationship between SFH and the bordered theory.
Given a 3-manifold $M$ with parameterized boundary, for a set of sutures $\g$, there
is an associated $\mathcal{A}(\partial M)$ module such that tensoring with it
one obtains $SFH(M,\g)$, see~\cite{bimod}. On the other hand, bordered Floer homology can also
be recovered from $SFH(M,\g)$ for all $\g$ and certain cobordism maps between these.

Heegaard Floer homology and monopole Floer homology are equivalent. One definite advantage
of the former is that it can be computed algorithmically. The breakthrough results
in this direction are due to Sarkar and Wang~\cite{Suc}, who gave
an algorithm for computing $\HFh(Y)$ for an arbitrary closed, connected, oriented 3-manifold~$Y$,
and to Manolescu, Ozsv\'ath, and Sarkar~\cite{MOS}, who gave a combinatorial characterization
of knot Floer homology $\HFK(S^3,K)$, where the input data is a grid diagram for~$K$.
The latter led to an invariant of Legendrian and transverse knots in contact
3-manifolds, cf.\ Ozsv\'ath, Szab\'o, and Thurston~\cite{OST}.

The hat version of $HF$ is considerably simpler to compute than the other flavors.
It took Manolescu, Ozsv\'ath, and Thurston~\cite{link-grid, MOT} several years to
bring the grid diagram approach to fruition and show that all flavors of~$HF$,
including the 4-manifold invariants, are algorithmically computable. The input data in these is
a surgery presentation of the 3-manifold, where the link on which we do integral surgery
is given by a grid diagram. The combinatorial theory of Heegaard Floer homology has grown
into a large area that we do not intend to cover here, instead, we refer the reader to
the survey article of Manolescu~\cite{grid-survey}. It is important to note that the above mentioned
algorithms are all far from being polynomial time and are unsuitable for even computing
the knot Floer homology of slightly larger knots. Also, so far, these theories have shed
very little light on the geometry of 3- and 4-manifolds. One notable result is due to
Sarkar~\cite{grid-tau}, which is the second completely combinatorial proof of the Milnor
conjecture on the slice genera of torus knots (following Rasmussen's proof~\cite{Milnor-conj} via Khovanov
homology).

Heegaard Floer homology has been really fruitful in the study of contact 3-manifolds,
and Legendrian and transverse knots. Given a contact 3-manifold $(Y,\xi)$, Ozsv\'ath
and Szab\'o~\cite{OSz4} associate to it an element
\[
c(\xi) \in \HFh(-Y)/\pm 1.
\]
This captures a lot of geometric information about the contact structure, as we shall see in the
following section.

\section{Applications}

The goal of this section is to showcase some of the many applications of
Heegaard Floer homology, with an emphasis on results whose statements are
purely topological and do not refer to~$HF$ itself.
We also list results that show that $HF$ contains very deep geometric information.

Let $g^*(K)$ denote the 4-ball genus of a knot $K \subset S^3$; i.e., the minimal
genus of a smooth oriented surface bounded by $K$ in $D^4$.
Currently, there is no algorithm known for computing $g^*(K)$, or even to
determine whether $g^*(K) = 0$ (such knots are called \emph{slice}).
For example, it is not known whether the 11-crossing Conway knot is slice
(whereas its mutant, the Kinoshita-Teresaka knot, is slice).
Using knot Floer homology, Ozsv\'ath and Szab\'o~\cite{OSz14} constructed an invariant $\tau(K) \in \ZZ$ such that
\begin{equation} \label{eqn:tau}
|\tau(K)| \le g^*(K).
\end{equation}
Moreover, $\tau$ descends to a group homomorphism from the concordance group of knots in~$S^3$ to $\ZZ$.
This allowed them to give another proof of the Milnor conjecture, originally
proved by Kronheimer and Mrowka~\cite{KM1} using Donaldson invariants. This states that for the torus
knot $T_{p,q}$, one has
\[
g^*(T_{p,q}) = \frac{(p-1)(q-1)}{2}.
\]
As alluded to in the previous section, Sarkar reproved inequality~\eqref{eqn:tau} using grid diagrams and
combinatorics, giving a purely elementary proof of the Milnor conjecture.
Note that Rasmussen introduced the $s$-invariant using Khovanov homology, which also gives
a lower bound on the 4-ball genus, as $|s(K)| \le 2g^*(K)$. Rasmussen conjectured that
$s(K) = 2\tau(K)$, which was then disproved by Hedden and Ording~\cite{HO}. Rasmussen~\cite{Milnor-conj}
gave the first purely combinatorial proof of the Milnor conjecture using his $s$-invariant.

There are many examples of knots~$K$ with trivial Alexander polynomial -- and due to
Freedman, these are topologically slice -- but for which $\tau(K) \neq 0$, and are hence not smoothly slice.
Such knots can be used to construct exotic smooth structures on $\RR^4$.

We already mentioned that knot Floer homology detects the Seifert genus and fibredness of
a knot. Ozsv\'ath and Szab\'o~\cite{OSz6} proved that $\HFh(Y)$ detects the Thurston norm of $Y$,
while Ni~\cite{fibred-3mfd} showed that it also detects fibredness of $Y$. It is an interesting question
how much geometric information is contained by the Heegaard Floer groups. For example, currently
no relationship is known between $\pi_1(Y)$ and $\HFh(Y)$. We state an important conjecture
that would make progress in this direction. But first, we need two definitions.

Note that for a rational homology 3-sphere $Y$, we always have
\[
\text{rk}\,\HFh(Y) \ge |H_1(Y)|.
\]
Indeed, for every $\spinc$-structure $\s \in \spinc(Y)$, the Euler characteristic of $\HFh(Y,\s)$
with respect to the absolute $\ZZ_2$-grading is~$1$ for
a rational homology 3-sphere. So $\text{rk}\, \HFh(Y,\s) \ge 1$ for every $\s \in \spinc(Y)$.
A rational homology 3-sphere $Y$ is called an \emph{L-space}
if
\[
\text{rk}\,\HFh(Y) = |H_1(Y)|;
\]
i.e., if its Heegaard Floer homology is as simple as possible.
This is equivalent to saying that for every $\spinc$-structure $\s \in \spinc(Y)$,
we have $\text{rk}\, \HFh(Y,\s) = 1$.
The terminology originates from the fact that every lens-space is an L-space.
However, there are many more: Ozsv\'ath and
Szab\'o~\cite{OSz12} showed that the double cover of~$S^3$ branched over any non-split alternating link is
an L-space. Furthermore, they proved~\cite{OSz5} that every 3-manifold with elliptic geometry is an L-space.
To extend the notion of L-spaces from rational homology 3-spheres to arbitrary 3-manifolds, we need to
look beyond the hat version of Heegaard Floer homology.
We say that a 3-manifold~$Y$ is an L-space if $HF^+_{\text{red}}(Y) = 0$.

A group $G$ is called \emph{left-orderable} if it is non-trivial, and it can be endowed with a linear order such that
if $g < f$ for $g$, $f \in G$, then $hg < hf$ for every $h \in G$.

\begin{conj} \label{conj:orderable}
Let $Y$ be an irreducible rational homology 3-sphere. Then the following three statements are equivalent.
\begin{enumerate}
\item \label{it:1} $Y$ is an $L$-space,
\item \label{it:3} $\pi_1(Y)$ is not left-orderable,
\item \label{it:2} $Y$ carries no taut foliation.
\end{enumerate}
\end{conj}

The conjecture that~\eqref{it:1} and~\eqref{it:3} are equivalent is due to Boyer, Gordon, and Watson~\cite{orderable}.
Ozsv\'ath and Szab\'o~\cite{OSz6} proved that~\eqref{it:1} implies~\eqref{it:2}.
Now we state another conjecture, originally due to Ozsv\'ath and Szab\'o, cf.\ Hedden and Ording~\cite{splicing}.

\begin{conj} \label{conj:3-sphere}
If $Y$ is an irreducible homology sphere that is an L-space, then $Y$ is homeomorphic to
either~$S^3$ or the Poincar\'e homology sphere.
\end{conj}

Observe that the implication \eqref{it:3} $\Rightarrow$ \eqref{it:1} in Conjecture~\ref{conj:orderable},
together with Conjecture~\ref{conj:3-sphere}, would imply
the \emph{Poincar\'e conjecture}. Indeed, if $Y$ is a simply-connected 3-manifold, then it is an irreducible
homology sphere. Since $\pi_1(Y) = 1$ is not left-orderable, $Y$ is an $L$-space by Conjecture~\ref{conj:orderable}.
Using Conjecture~\ref{conj:3-sphere}, we get that $Y$ is homeomorphic to~$S^3$, as the Poincar\'e homology sphere
is not simply-connected.

Heegaard Floer homology has been particularly successful in tackling problems
on Dehn surgery. The main tool is the following surgery exact triangle.

\begin{theorem}
Let $K$ be a knot in the closed oriented 3-manifold $Y$, together with
framings $f$, $g \in H_1(\partial N(K))$ such that $m \cdot f = f \cdot g = g \cdot m = 1$,
where $m$ denotes the class of the meridian. Then there is an exact sequence
\[
\dots \longrightarrow \HFh(Y) \longrightarrow \HFh(Y_f(K)) \longrightarrow \HFh(Y_g(K)) \longrightarrow \dots
\]
\end{theorem}

In particular, when $Y$ is a homology 3-sphere, then we can take $Y_f(K) = Y_0(K)$ and $Y_g(K) = Y_1(K)$.

The following theorem was originally proved by Kronheimer, Mrowka, Ozsv\'ath, and Szab\'o~\cite{KMOS} using
monopole Floer homology, but now we know that is isomorphic to~$HF$.

\begin{theorem}
If the knot $K \subset S^3$ is not the unknot, then
\[
S^3_{p/q}(K) \not\approx \RR P^3.
\]
\end{theorem}

More generally, they obtain the following.

\begin{theorem}
Let $K$ be a knot in $S^3$.
If there is an orientation preserving homeomorphism between $S^3_{p/q}(K)$ and the lens-space $S^3_{p/q}(U)$, then $K = U$.
\end{theorem}

The lens space realization problem, asking which lens-spaces can be obtained by Dehn-surgery along
a non-trivial knot in~$S^3$, has recently been settled via Heegaard Floer homology by Greene~\cite{Greene}.
Note that the cyclic surgery theorem of Culler, Gordon, Luecke, and Shalen~\cite{cyclic-surgery} ensures that if a surgery on
a non-trivial knot yields a lens space, then the surgery coefficient has to be an integer.
It is still an open question of Berge exactly which knots yield lens space surgeries. These are conjectured to be
the doubly primitive knots that were classified by Berge. What Greene showed is that the lens spaces
that can be realized by surgery on a non-trivial knot are exactly the ones obtained by surgery on a Berge
knot. He also exhibits that such a knot has to have the same knot Floer homology as a Berge knot.

Another area where Heegaard Floer homology has been very successful is deciding whether a knot
has unknotting number one. Given a knot~$K$ in~$S^3$, we denote by~$u(K)$ its \emph{unknotting number},
which is the minimal number of times $K$ intersects itself during a regular homotopy to the unknot.
Currently, no algorithm is known for computing $u(K)$, or even to decide whether $u(K) = 1$.
Of course, it is easy to give an upper bound on $u(K)$ by exhibiting a concrete unknotting sequence.
A classical lower bound is provided for example by the knot signature.
It is easy to prove that $g^*(K) \le u(K)$, and hence $|\tau(K)|$ also provides a lower bound.
However, $g^*(K)$ can be zero, while $u(K)$ is large.
Note that recent work of Lackenby (in progress) gives an algorithm for telling whether a non-split link with at least
two components and satisfying some mild restrictions has \emph{unlinking number} one.

The main tool for detecting unknotting one knots is the \emph{Montesinos trick}.
Given a knot~$K$ in~$S^3$, let $\Sigma(K)$ denote the double cover of~$S^3$,
branched along~$K$. If $u(K) = 1$, then there exists a knot~$C$ in~$S^3$ and an integer $n$
such that
\[
\Sigma(K) = S^3_{n/2}(C).
\]
Now the surgery exact triangle can be used to give an obstruction to $\Sigma(K)$
being a half-integral Dehn surgery on a knot. Using this method,
Ozsv\'ath and Szab\'o~\cite{OSz11} showed that, for example, $u(8_{10}) = 2$.

The links~$L$ and~$L'$ in~$S^3$ are related by \emph{Conway mutation} if
there exists an embedded 2-sphere $S \subset S^2$ that intersects~$L$ in
four points $L \cap S = L' \cap S$, and we can obtain~$L'$ from $L$ by removing the ball bounded by~$S$ and
regluing it using an involution that fixes two points on~$S$. Viro showed that if
$L$ and $L'$ are mutants, then $\Sigma(L) \approx \Sigma(L')$. Greene~\cite{mutation} showed that the converse
also holds for alternating links. More precisely, he proved that the following are equivalent:
\begin{enumerate}
\item \label{it:mutants} $L$ and $L'$ are mutants,
\item \label{it:cover} $\S(L) \approx \S(L')$,
\item $\HFh(\S(L)) \cong \HFh(\S(L'))$ as absolutely graded, relatively $\Spinc$-graded groups.
\end{enumerate}
The proof that~\eqref{it:mutants} and~\eqref{it:cover} are equivalent passes through $HF$ in an essential way.

The correction terms $d(Y,\s)$, introduced by Ozsv\'ath and Szab\'o~\cite{OSz15},
have numerous important applications.
For example, they provide alternate proofs of Donaldson's diagonalizability
theorem and the Thom conjecture for $\CC P^2$.
Due to lack of space, we do not pursue this 4-manifold topological direction any further.

Finally, we mention some of the results in contact topology obtained using $HF$.
First a few definitions.
Given a contact 3-manifold $(Y,\xi)$, we say that $\Delta \subset Y$ is an
\emph{overtwisted disk} if $\xi$ is tangent to $\Delta$ along $\partial \Delta$.
The contact structure $\xi$ is called \emph{overtwisted} if it contains an
overtwisted disk, and is \emph{tight} otherwise. By a result of Eliashberg,
overtwisted contact structures satisfy an h-principle, and there is a unique
one up to isotopy in each homotopy class of oriented 2-plane fields.
In general, it is very difficult to classify tight contact structures on 3-manifolds.

We say that $(Y,\xi)$ is \emph{Stein fillable} if $Y$ bounds a complex surface $S$
such that the two-plane field~$\xi$ consists of complex lines. Furthermore, $S$ admits a Morse function $f \colon S \to I$
with $Y = \partial S = f^{-1}(1)$, and for every regular value $t$ of $f$,
the complex tangencies to $f^{-1}(t)$ form a contact structure. If $(Y,\xi)$ is
Stein fillable, then it is necessarily tight.

Given a contact 3-manifold $(Y,\xi)$, Ozsv\'ath and Szab\'o~\cite{OSz4}
assign to it an invariant $c(\xi) \in \HFh(-Y)$ using open book decompositions.
They proved that if $(Y,\xi)$ is overtwisted, then $c(\xi) = 0$.
In other words, $c(\xi)$ can be used to detect tightness. However, the
converse is not true, as shown by Ghiggini~\cite{tight, fillability}: there are tight contact
structures with $c(\xi) = 0$. But at least we have the following non-vanishing result of
Ozsv\'ath and Szab\'o: if $\xi$ is Stein fillable, then $c(\xi) \neq 0$.

The above result has been strengthened by Ghiggini~\cite{fillability}. We say that $(Y,\xi)$
is \emph{strongly symplectically fillable} if there is a contact form $\a$
such that $\xi = \ker(\a)$ and a symplectic 4-manifold $(X,\omega)$
such that $\partial X = Y$ and $\omega|_Y = d\a$. For example, every Stein
fillable contact 3-manifold is strongly symplectically fillable.
Furthermore, the contact 3-manifold $(Y,\xi)$ is called \emph{weakly symplectically fillable} if there is a
symplectic manifold $(X,\omega)$ such that $\partial X = Y$ and $\omega|_{\xi} > 0$.
Every strongly fillable contact manifold is weakly fillable, and
every weakly fillable contact manifold is tight.
Ghiggini showed that if $\xi$ is strongly symplectically fillable,
then $c(\xi) \neq 0$. He also proved that for a particular family $\{(Y,\xi_n) \colon n \in \ZZ_+\}$ of
weakly symplectically fillable contact 3-manifolds constructed by Giroux,
one has $c(\xi_n) = 0$ for every $n \in \ZZ_+$. In particular, each $(Y,\xi_n)$ is weakly fillable (and hence tight),
but not strongly fillable (and hence not Stein fillable).
It was also Ghiggini~\cite{strong-Stein} who constructed the first family of strongly fillable but
not Stein fillable contact structures, and this work relies on Heegaard Floer homology as well.

Using the contact invariant, Lisca and Stipsicz~\cite{LS}
proved that every closed, oriented Seifert fibred 3-manifold carries a tight contact structure,
except if it arises as $(2n-1)$-surgery on the torus knot $T_{2,2n+1}$.
Furthermore, Ghiggini, Lisca, and Stipsicz~\cite{GLS} used the contact invariant
to classify all tight contact structures on certain small Seifert fibred spaces.

\section{Lagrangian Floer homology}

The machinery underlying the construction of Heegaard Floer homology is Lagrangian
intersection Floer theory, first introduced by Floer himself
\cite{Floer}.  We will follow the threatment of Fukaya et
al.~\cite[Chapter~2]{FOOO}.

First, recall that a \emph{symplectic $\RR$-vector space}
is a pair $(V,\omega)$, where $V$ is an
$\RR$-vector space and $\omega$, called the \emph{symplectic form},
is a non-degenerate anti-symmetric
bilinear form on $V$. The existence of such an~$\omega$
implies that $V$ is necessarily of some even dimension $2n$.  A
subspace $W$ of $V$ is called \emph{Lagrangian} if it is of dimension $n$, and
$\omega(v,w) = 0$ for every $v$, $w \in W$. If $V = \RR^{2n}$ with
standard basis $(v_1,\dots,v_{2n})$, then there is a canonical
symplectic form $\omega_0$ such that $\omega_0(v_k,v_{n+k}) = 1$ and
$\omega_0(v_{n+k},v_k) = -1$ for every $k \in \{\,1,\dots, n \,\}$,
and such that~$\omega_0$ vanishes for every other pair of basis
vectors.  We denote the Grassmannian of Lagrangian subspaces in
$(\RR^{2n},\omega_0)$ by~$\Lambda(n)$. Note that $\Lambda(n) =
U(n)/O(n)$, and $\pi_1(\Lambda(n)) = \ZZ$.

A \emph{symplectic manifold} is a pair $(M,\omega)$ such that $M$ is a
$2n$-manifold for some positive integer $n$, and $\omega$ is a closed
non-degenerate differential 2-form on $M$; i.e., $d\omega = 0$ and
$\omega^n$ is nowhere zero. In other words, $(T_pM, \omega_p)$ is a
symplectic vector space for every point $p \in M$ and $d\omega = 0$.
We say that $L$ is a \emph{Lagrangian submanifold} of $(M,\omega)$ if $L$ is
a smooth submanifold of $M$ such that $T_pL$ is a Lagrangian subspace
of $(T_pM,\omega_p)$ for every $p \in L$.

Note that every symplectic vector bundle over $S^1$ is symplectically
trivial; i.e., isomorphic to the product $S^1 \times \RR^{2n}$ with
symplectic form $\omega_0$ on each fiber.  Hence, given a curve
$\gamma \colon S^1 \to L$, the bundle $\gamma^*TM$ is symplectically
trivial; let
\[
\Phi \colon \gamma^* TM \to S^1 \times \RR^{2n}
\]
be one such trivialization. We denote by $p$ the projection from $S^1 \times
\RR^{2n} \to \RR^{2n}$.  For every $\tau \in S^1$, consider the
Lagrangian subspace $p \circ \Phi(T_{\gamma(\tau)}L)$ of
$(\RR^{2n},\omega_0)$, this is an element of $\Lambda(n)$.  So we have
obtained a loop $S^1 \to \Lambda(n)$ whose homotopy class, considered
as an element of $\ZZ$, is called the Maslov index of $\gamma$ in the
trivialization~$\Phi$, and is denoted by $\mu_\Phi(\gamma)$.

This gives rise to the Maslov index homomorphism
\[
\mu_L \colon \pi_2(M,L) \to \ZZ.
\]
Given a continuous map $u \colon (D^2,\partial D^2) \to (M,L)$,
we let $\mu_L([u]) = \mu_{\Phi}(\g)$, where $\g= u|_{\partial D^2}$,
and $\Phi$ is the trivialization of~$\gamma^* TM$ that extend to~$u^* TM$.

Suppose we have a symplectic $2n$-manifold $(M,\omega)$, together
with two compact and connected Lagrangian submanifolds $L_0$ and
$L_1$ that intersect transversely.
Then fix a path $\ell_0$ from $L_0$ to $L_1$. Let
$\Omega(L_0,L_1)$ be the space of paths $\ell \colon I \to M$ such
that $\ell(0) \in L_0$ and $\ell(1) \in L_1$, and we denote by
$\Omega_{\ell_0}(L_0,L_1)$ the component of $\ell_0$.  The universal
covering space of $\Omega_{\ell_0}(L_0,L_1)$ is homeomorphic to the
space of pairs $(\ell,[w])$, where $\ell \in \Omega_{\ell_0}(L_0,L_1)$
and $[w]$ is a homotopy class of maps $w \colon I \times I \to M$ such
that $w(0,t) = \ell_0(t)$ and $w(1,t) = \ell(t)$ for every $t \in I$,
while $w(\tau,0) \in L_0$ and $w(\tau,1) \in L_1$ for every $\tau \in
I$.

Next, we define another covering space of $\Omega_{\ell_0}(L_0,L_1)$.
Let $c \colon S^1 \to \Omega(L_0,L_1)$ be a closed loop. This can be
thought of as a map $c \colon S^1 \times [0,1] \to M$ such that $c_0 =
c(\cdot,0)$ is a curve on $L_0$ and $c_1 = c(\cdot,1)$ is a curve on
$L_1$.  Since every symplectic vector bundle over $S^1$ is trivial,
the bundle $c^* TM$ is symplectically trivial. Fix such a
trivialization $\Psi \colon c^*TM \to S^1 \times I \times \RR^{2n}$,
and let $\Psi_i$ be the induced trivialization of $c_i^*TM$ for $i \in
\{0,1\}$.  Then one can show that $\mu_{\Psi_0}(c_0) -
\mu_{\Psi_1}(c_1)$ is independent of the choice of trivialization
$\Psi$ and only depends on the homotopy class of the loop~$c$; we
denote it by $\mu(c)$. In fact, it defines a homomorphism
\[
\mu \colon \pi_1(\Omega_{\ell_0}(L_0,L_1), \ell_0) \to \ZZ.
\]

Given two paths $w$ and $w'$ in $\Omega_{\ell_0}(L_0,L_1)$ from
$\ell_0$ to $\ell$, we denote by $w \# \ol{w}'$ the concatenation of
$w$ with the reverse of $w'$. This is a closed loop in
$\Omega_{\ell_0}(L_0,L_1)$ based at $\ell_0$. We say that $(\ell,w)$
and $(\ell,w')$ are equivalent, and write $(\ell,w) \sim (\ell, w')$, if
$\int_w \omega = \int_{w'} \omega$; i.e., if $\int_{w \# \ol{w}'} \omega
= 0$, and if $\mu(w \# \ol{w}') = 0$.  We denote by $[\ell,w]$ the
equivalence class of $(\ell,w)$.  Then the space
$\widetilde{\Omega}_{\ell_0}(L_0,L_1)$ is defined to be the set of
equivalence classes of pairs $(\ell,w)$; this is a covering space of
$\Omega_{\ell_0}(L_0,L_1)$.

Loosely speaking, Lagrangian intersection Floer homology of the pair
$(L_0,L_1)$ is the homology of the path space
$\widetilde{\Omega}_{\ell_0}(L_0,L_1)$. One can compute the homology
of a finite dimensional manifold $M$ using Morse theory: First, one
has to choose a Morse function $f$ and a Riemannian metric on $M$
whose gradient satisfies the Morse-Smale condition. The chain complex
is freely generated by the critical points of $f$, and the
differential counts gradient flow-lines between critical points of
index difference one. However, on an infinite dimensional manifold, it
does not make sense to talk about the Morse index, and there are other
technical difficulties.  It was Floer's observation that in our case
it suffices to define the index difference for a pair of critical
points, and one can still obtain homology groups under certain
hypothesis on the pair $(L_0,L_1)$. These hypothesis were relaxed by
Fukaya et al.~\cite{FOOO}, but in some cases the differential does not
square to zero, and the theory is obstructed in an essential way.

For our ``Morse function,'' we take the functional $\cA = \cA_{\ell_0}
\colon \widetilde{\Omega}_{\ell_0}(L_0,L_1) \to \RR$, given by the
formula
\[
\cA([\ell,w]) = \int w^* \omega.
\]
The critical points of $\cA$ correspond to the intersection points of
$L_0$ and $L_1$; i.e., they are of the form $[\ell_p,w]$, where
$\ell_p \colon I \to M$ is the constant path at some $p \in L_0 \cap
L_1$.  We denote by $C(\cA)$ the set of critical points of $\cA$.

To define a metric on $\wt{\Omega}_{\ell_0}(L_0,L_1)$, we first choose
a $t$-dependent family $J = \{\, J_t \colon t \in I \,\}$ of almost
complex structures on $M$ compatible with $\omega$, and consider the
induced family of Riemannian metrics $g_J = \omega(\cdot, J \cdot)$ on
$M$. This will, in turn, induce an $L^2$-metric on
$\wt{\Omega}_{\ell_0}(L_0,L_1)$ by
\[
\langle \xi, \eta \rangle = \int_0^1 g_{J_t}(\xi(t),\eta(t)) \, dt,
\]
where $\xi$ and $\eta$ are vector fields along some path $\ell \in \Omega_{\ell_0}(L_0,L_1)$.

It turns out that the moduli space of gradient flow-lines of $\cA$
between critical points $[\ell_p,w]$ and $[\ell_q,w']$ is the set
$\cM_J([\ell_p,w],[\ell_q,w'])$ of maps $u \colon \RR \times I \to M$
such that
\begin{enumerate}
\item $u(\RR \times \{j\}) \subset L_j$ for $j \in \{0,1\}$,
\item $\partial u/ \partial \tau + J(t,u) \partial u / \partial t = 0$ (i.e.; $u$ is \emph{J-holomorphic}),
\item $\lim_{\tau \to -\infty} u(\tau,t) = p$ and $\lim_{\tau \to +
    \infty} u(\tau, t) = q$,
\item $w \# u \sim w'$, where the concatenation is taken along
  $\ell_p$.
\end{enumerate}

Note that the $t$-dependence of $J$ is to ensure that the above moduli
spaces are cut out transversely, and are hence smooth manifolds.  From
now on, we will suppress~$J$ in the notation whenever possible.
Observe that $\RR \times I$ is conformally equivalent with $D^2 \setminus \{\pm i\}$.
So, instead, one can consider $J$-holomorphic Whitney
disks $u \colon D^2 \to M$ connecting~$p$ and~$q$ (i.e,
$u(\partial D^2 \cap \{\Re \le 0\}) \subset L_0$, $u(\partial D^2 \cap \{\Re \ge 0\})
\subset L_1$, $u(-i) = p$, and $u(i) = q$)
in the relative homotopy class given by $w \# u \sim w'$.
This is often the viewpoint taken in Heegaard Floer homology.

We next define a consistent grading on $C(\cA)$. For this, we have to
fix a section~$\lambda^0$ of $\ell_0^* \Lambda M$, where $\Lambda M$
is the bundle of Lagrangian Grassmanians of $TM$, such that
$\lambda^0(0) = T_{\ell_0(0)}L_0$ and $\lambda^0(1) = T_{\ell_0(1)}
L_1$.  Let $[\ell_p,w] \in C(\cA)$ be an element corresponding to the
intersection point $p \in L_0 \cap L_1$. Choose a trivialization
\[
\Phi \colon w^* TM \to I^2 \times \RR^{2n} \to \RR^{2n}
\]
such that
\begin{itemize}
\item $\Phi(0,t) \circ w^*(\lambda^0(t)) \equiv \RR^n$,
\item $\Phi(1,0) \circ w^*(T_pL_0) = \RR^n$, and
\item $\Phi(1,1) \circ w^*(T_pL_1) = i \cdot\RR^n$.
\end{itemize}
This $\Phi$ will induce a loop $\lambda_\Phi \colon \partial I^2 \to
\Lambda(n)$, such that
\begin{itemize}
\item $\lambda_\Phi(s,0) = \Phi \circ w^*(T_{w(s,0)} L_0)$,
\item $\lambda_\Phi(1,t) = e^{\frac{\pi i t}{2}}
  \RR^n$,
\item $\lambda_\Phi(s,1) = \Phi \circ w^*(T_{w(s,1)} L_1)$, and
\item $\lambda_\Phi(0,t) = \RR^n$
\end{itemize}
for every $s$, $t \in I$.
We define the degree of $[\ell_p,w]$ by
\[
\mu_{\lambda_0}([\ell_p,w]) = \mu(\lambda_\Phi).
\]
The above number is independent of the choice of trivialization
$\Phi$.  So we obtain a grading on $C(\cA)$, provided the based path
$\ell_0$ and the section $\lambda_0$ of $\Lambda M$ over $\ell_0$ are
chosen.

We say that the pair of Lagrangians $(L_0,L_1)$ is \emph{relatively spin}
if there is a class $st \in H^2(M;\ZZ_2)$ such that $st|_{L_i} = w_2(L_i)$
for $i = 0$, $1$. A \emph{relative spin structure} on $(L_0, L_1)$
consists of a class $st \in H^2(M ;\ZZ_2)$, an oriented vector bundle~$\xi$ over the 3-skeleton
of~$M$ such that $w_2(\xi) = st$, and $\text{Spin}$-structures on $TL_i \oplus \xi$
over the 2-skeleton of $L_i$.

The following transversality result was obtained by Floer~\cite{Floer}, with some
missing details filled in by Oh~\cite{Oh3, Oh}.

\begin{theorem}
  Let $[\ell_p,w]$, $[\ell_q,w']$ be critical points of the functional
  $\cA$.  Then the space $\cM([\ell_p,w],[\ell_q,w'])$ is a smooth
  manifold of dimension $\mu([\ell_p,w]) - \mu([\ell_q,w'])$. If we
  also assume that the pair $(L_0,L_1)$ is relatively spin, then the
  space will carry an orientation.
\end{theorem}


Since $J = \{J_t\}$ does not depend on $\tau$, there is a natural
$\RR$-action on the moduli space $\cM([\ell_p,w],[\ell_q,w'])$ defined by translating
along the $\tau$-direction. We put
\[
\tcM([\ell_p,w],[\ell_q,w']) = \cM([\ell_p,w],[\ell_q,w'])/\RR.
\]
A Lagrangian~$L$ is called \emph{monotonic} if the symplectic area of every
pseudo-hol\-o\-mor\-phic disk with boundary on~$L$ is positively proportional to
the Maslov index~$\mu_L$. From now on, we assume that $L_0$ and $L_1$ are monotonic.
When the moduli space $\tcM([\ell_p,w],[\ell_q,w'])$
is 0-dimensional; i.e., when $\mu([\ell_p,w]) - \mu([\ell_q,w']) = 1$,
then it is a compact oriented 0-dimensional manifold, hence we can
consider the algebraic count $\# \tcM([\ell_p,w],[\ell_q,w']) \in \ZZ$ of its points.

\begin{definition}
The Floer chain complex $(CF(L_0,L_1;\ell_0),\partial_0)$ is defined as
\[
CF^k(L_0,L_1) = CF^k(L_0,L_1; \ell_0) =  \widehat{\bigoplus}_{[\ell_p,w] \in C(\cA) \atop \mu([\ell_p,w]) = k} \QQ[\ell_p,w],
\]
where $\widehat{\bigoplus}$ means a certain algebraic completion,
and the differential of $[\ell_p,w] \in CF^k(L_0,L_1)$ is given by
\[
\partial_0[\ell_p,w] = \sum_{[\ell_q,w'] \in CF^{k-1}(L_0,L_1)} \# \tcM([\ell_p,w],
[\ell_q,w']) \cdot [\ell_q,w'].
\]
\end{definition}

We briefly explain the completion used in the above definition.
The elements of $CF^k(L_0,L_1)$ are the (infinite) sums
\[
\sum_{[\ell_p,w]} a_{[\ell_p,w]}[\ell_p,w]
\]
such that $a_{[\ell_p,w]} \in \QQ$, and for each $C \in \RR$, the set
\[
\{\, [\ell_p,w] \colon \mathcal{A}(\ell_p,w) \le C,\, a_{[\ell_p,w]} \neq 0 \,\}
\]
is finite.

Let $G(L_0,L_1)$ be the group of deck transformations
of the covering
\[
\wt{\Omega}_{\ell_0}(L_0,L_1) \to \Omega_{\ell_0}(L_0,L_1).
\]
Then $CF(L_0,L_1)$ is a module over the Novikov ring $\Lambda(L_0,L_1)$, where
$\Lambda^k(L_0,L_1)$ is the set of all (infinite) sums
$\sum_{g \in G(L_0,L_1), \, \mu(g) = k} a_g [g]$ such that $a_g \in \QQ$,
and for each $C \in \RR$, the set
\[
\{\, g \in G(L_0,L_1) \colon E(g) \le C, \, a_g \neq 0 \,\}
\]
is finite.

The minimal Maslov number of a Lagrangian~$L$ is the non-negative integer~$\S_L$
defined by~$\text{Im}(\mu_L) = \S_L \ZZ$, where $\mu_L \colon \pi_2(M,L) \to \ZZ$
is the Maslov index homomorphism.
Building on work of Floer~\cite{Floer}, Oh~\cite{Oh3, Oh4} proved that
$\partial_0 \circ \partial_0 = 0$ for monotone Lagrangian
submanifolds, where the minimal Maslov number is bigger than two, with
some topological restrictions on the pair $(L_0,L_1)$. As in the case of Morse
homology, the proof relies on showing that the ends of the compactification of
1-dimensional moduli spaces $\cM([\ell_p,w],[\ell_q,w'])$ are only
broken flow-lines.

The reader might be wondering why one has to consider the covering space $\wt{\Omega}_{\ell_0}(L_0,L_1)$
instead of trying to do Morse theory on $\Omega_{\ell_0}(L_0,L_1)$. If one did that,
one would run into the problem of having infinitely many pseudo-holomorphic disks (gradient flow-lines
on the path space) between two intersection points (critical points of the functional), and
the boundary map would not be defined. If one instead passed to the whole universal cover,
this problem would be resolved, and
we would get the twisted (co)homology of $\Omega_{\ell_0}(L_0,L_1)$ with coefficients
lying in some completion of $\ZZ[\pi_1(\Omega_{\ell_0}(L_0,L_1))]$.
This is unfortunate since the coefficient ring depends on the pair of Lagrangians.
Since there are only finitely many disks with bounded
symplectic area, it suffices to pass to the cover $\wt{\Omega}_{\ell_0}(L_0,L_1)$,
and we can work over a Novikov ring independent of the Lagrangians.
Twisting with the Maslov index allows one to define a homological $\ZZ$-grading,
solving the issue of having disks between two intersection points of different Maslov indices
(there would still be a $\ZZ_2$-grading in this case).

\section{Definition of Heegaard Floer homology}

\begin{definition}
  Let $Y$ be a closed, connected, oriented 3-manifold. A \emph{Heegaard diagram}
  for $Y$ is a triple $(\S,\alphas,\betas)$ such that
  \begin{itemize}
  \item $\S$ is a closed oriented genus $g$ surface embedded in $Y$,
  \item $\alphas \subset \S$ is a properly embedded 1-manifold with
    components $\a_1,\dots,\a_g$ that are linearly independent in
    $H_1(\S)$, and each of which bounds a disk in $Y$ to the negative
    side of $\S$,
  \item $\betas \subset \S$ is a properly embedded 1-manifold with
    components $\b_1, \dots, \b_g$ that are linearly independent in
    $H_1(\S)$, and each of which bounds a disk to the positive side of
    $\S$.
  \end{itemize}
\end{definition}

A Heegaard diagram of $Y$ arises from a self-indexing Morse function
$f \colon Y \to \RR$ with a single minimum and maximum, together with
a Riemannian metric on $Y$. If $p$ is a critical point of $f$, then we
denote by $W^s(p)$ and $W^u(p)$ the stable and unstable manifolds of
$p$ under the gradient flow of~$f$.
We take $\S$ to be $f^{-1}(3/2)$. The $\a$-curves are of the form
$\S \cap W^u(p)$, where $p$ is an index-1 critical point of $f$, and
the $\b$-curves are of the form $W^s(q) \cap \S$, where $q$ is an
index-2 critical point of~$f$.

Let $\H = (\S,\alphas,\betas)$ and $\H' = (\S',\alphas',\betas')$ be
Heegaard diagrams of $Y$. Then we say that $\H'$ is obtained from $\H$
by a \emph{stabilization} if $\S' \setminus \S$ is a punctured torus,
and $\alphas' = \alphas \cup \a$ and $\betas' = \betas \cup \b$, where
$\a$ and $\b$ are two curves in $\S' \setminus \S$ that intersect each
other transversely in a single point. In this case, we also say that~$\H$
is obtained from~$\H'$ via a \emph{destabilization}.

The diagram $\H'$ is obtained
from $\H$ by \emph{handlesliding} $\a_i$ over $\a_j$ along the arc $a
\subset \S$ connecting $\a_i$ and $\a_j$ if $\S = \S'$ and $\betas =
\betas'$, while $\alphas' = (\alphas \setminus \a_i) \cup \a_i'$,
where $\a_i'$ is the boundary component of a thin regular neighborhood
$N(\a_i \cup a \cup \a_j)$ distinct from $\a_i$ and $\a_j$.

According to a theorem of Reidemeister and
Singer~\cite{Reidemeister, Singer}, any two Heegaard
diagrams of the 3-manifold~$Y$ become diffeomorphic after a sequence
of (de)stabilizations and handleslides. Slightly more is true, any two
diagrams become isotopic in~$Y$ after a suitable sequence of
(de)stabilizations and handleslides. However, this isotopy is far from
unique, even homotopically. This will cause some serious difficulties
when trying to prove the naturality of Heegaard Floer homology, and the failure
of naturality for the hat version.

For the construction of Heegaard Floer homology, it is essential to
fix a basepoint $z \in \S \setminus (\alphas \cup \betas)$.
Furthermore, we choose a complex structure $\mathfrak{j}$ on $\S$.
Consider the symmetric product
\[
\text{Sym}^g(\S) = \overbrace{\S \times \dots \times \S}^g / S_g,
\]
this is the space of unordered $g$-tuples of points in $\S$.
To see that this is a manifold, observe that $\sym^g(\CC) \cong \CC^g$
by the fundamental theorem of algebra: Given an ordered $g$-tuple
$(a_1, \dots, a_g) \in \CC^g$, the roots of the polynomial
$z^g + a_1 z^{g-1} + \dots + a_g$ give an unordered $g$-tuple.
Note that the smooth
structure on $\text{Sym}^g(\S)$ depends on $\mathfrak{j}$, but the
diffeomorphism type is independent.  The complex structure
$\mathfrak{j}$ induces a complex structure
$\text{Sym}^g(\mathfrak{j})$ on the symmetric product
$\text{Sym}^g(\S)$.  Let
\[
\T_\a = \a_1 \times \dots \times \a_g / S_g,
\]
and similarly,
\[
\T_\b = \b_1 \times \dots \times \b_g / S_g,
\]
these tori are totally real
submanifolds of $\text{Sym}^g(\S)$; i.e., they contain no complex
lines.  As shown by Ozsv\'ath and Szab\'o, one can study intersection
Floer homology for the pair $(\T_\a,\T_\b)$ in a
manner analogous to the Lagrangian case. If one chooses a symplectic
structure on $\S$, the induced symplectic structure on $\text{Sym}^g(\S)$
is singular. Perutz~\cite{Perutz} constructed a symplectic structure
on $\text{Sym}^g(\S)$ compatible with the complex structure that
makes $\T_\a$ and $\T_\b$ Lagrangian and monotonic with minimal Maslov
number two. Furthermore, he proved that a handleslide
can be realized by a Hamiltonian isotopy, however, the pair might cease to
be monotonic in the meantime (this is related to admissibility that we discuss next).

Monotonicity of $\T_\a$ and $\T_\b$ suffices to define the boundary map~$\partial$
for the Lagrangian intersection Floer homology of the pair $(\T_\a,\T_\b)$ in $\sym^g(\S)$, but since the minimal
Maslov number is two, one cannot simply refer to the work of Floer~\cite{Floer} and Oh~\cite{Oh3}
to claim that $\partial^2 = 0$ in this setting. However, Ozsv\'ath and Szab\'o~\cite[Theorem~3.15]{OSz1}
prove that in the boundary of the compactification of the one-dimensional moduli spaces,
the algebraic count of configurations that are not broken flow-lines
is zero (these are so called boundary bubbles, and they might actually appear), and hence indeed we have~$\partial^2 = 0$.
In the language of Fukaya et al.~\cite{FOOO}, the Lagrangian intersection
Floer homology of the pair $(\T_\a,\T_\b)$ is unobstructed.

However, the intersection Floer homology of $(\T_\a,\T_\b)$ is \emph{not} an invariant of the underlying
3-manifold.  Indeed, consider the two diagrams of $S^1\times S^2$ in Figure~\ref{fig:admissibility}.
\begin{figure}
\includegraphics{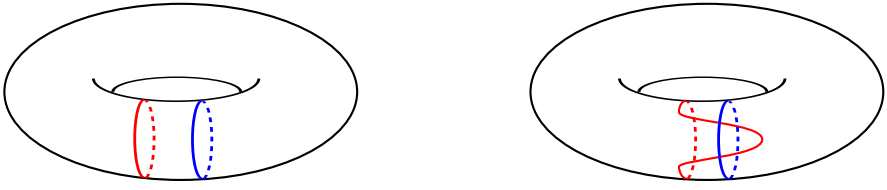}
\caption{A non-admissible diagram of $S^1 \times S^2$ on the left, and an admissible one on the right.}
\label{fig:admissibility}
\end{figure}
Both of them are genus one, so there is a single
$\alpha$-curve $\a$ and a single $\beta$-curve $\b$, and both of them
are meridians.  In the first diagram, $\a$ and $\b$ are disjoint,
hence the Floer homology group is generated by the empty set, in the
second, they intersect in two points transversely and the boundary map
is zero.  The two homologies have different ranks.

To make the construction work, Ozsv\'ath and Szab\'o introduced two
additional ingredients. One is the choice of a basepoint $z$ in $\S
\setminus (\alphas \cup \betas)$, and we record the intersection
number of every holomorphic disk with the subvariety
\[
V_z = \{z\} \times \text{Sym}^{g-1}(\S).
\]
More precisely, we introduce
an additional formal variable $U$, and deform the Floer complex
according to the intersection number with $V_z$. There are
several different ways of doing this, giving rise to the
different flavors of Heegaard Floer homology. Cf. Seidel~\cite{Seidel},
who first introduced such an idea in a different context.
Without fixing the basepoint, the invariant obtained would only
capture the algebraic topology of $Y$. Furthermore, it would not be natural, as
illustrated by~\cite[Example~3.3]{naturality}.

The second ingredient is that, when $b_1(Y) > 0$, we only consider so called \emph{admissible}
diagrams. In the case of~$\widehat{HF}$ and $HF^+$, we use weakly admissibile diagrams, but
to define $HF^\infty$ and $HF^-$, the diagram has to be strongly admissible
in reference to a fixed $\text{Spin}^c$-structure on~$Y$.
We will discuss admissibility later.

In fact, Heegaard Floer homology is an invariant of a $\Spinc$
3-manifold. As explained in the previous section, Lagrangian intersection Floer
homology is Morse theory on a \emph{component} of
the space $\Omega(\T_\a,\T_\b)$ of paths connecting $\T_\a$ and $\T_\b$. The choice
of basepoint~$z$ sets up a bijection between the components of
this path space with $\Spinc$-structures on $Y$. We explain this construction
next.

The group $\spinc(n)$ is defined to be
\[
\left( \text{Spin}(n) \times S^1 \right) / \, \langle(-1,-1)\rangle;
\]
i.e., it is the complexification
of the group $\text{Spin}(n)$. It fits into the exact sequence
\[
1 \rightarrow S^1 \rightarrow \spinc(n) \rightarrow SO(n) \rightarrow 1.
\]
The first map is given by $z \mapsto [1, z]$, while the
second is $[g, z] \mapsto p(g)$, where $p \colon \text{Spin}(n) \to SO(n)$
is the covering map.

Given an oriented Riemannian $n$-manifold~$M$, a $\spinc$-structure
on $M$ is the reduction/lift of the structure group of $TM$, the principal $SO(n)$-bundle $P_M$
of oriented orthonormal frames on~$M$, to $\spinc(n)$, considered up to equivalence.
For us, an equivalent description of $\Spinc$-structures on 3-manifolds due
to Turaev~\cite{spinc} will be more relevant, which we now review.
Let $Y$ be a closed, connected, oriented 3-manifold. We say
that the nowhere vanishing vector fields $v$ and $w$ on~$Y$ are
\emph{homologous}, and we write $v \sim w$, if they are homotopic in
the complement of a ball.  A $\Spinc$-structure on the 3-manifold $Y$ is
the homology class of nowhere vanishing vector fields, and we denote
the set of these by $\Spinc(Y)$. Note that we can think of a
$\Spinc$-structure as a homotopy class of nowhere vanishing vector
fields over the 2-skeleton of $Y$ that extend to~$Y$, and the obstruction to homotoping
$v$ to $w$ over the 2-skeleton (or equivalently, in the complement of
a ball), is an element of $H^2(Y) \cong H_1(Y)$. This way, we can
view $\Spinc(Y)$ as an affine space over $H_1(Y)$. The concrete
action of a class $a \in H_1(M)$ on $\s \in \Spinc(Y)$ is given by
Reeb turbularization: we represent~$a$ as an embedded oriented 1-manifold,
and then ``turbularize'' the vector field $v$ along it.
A homomorphism
\[
c_1 \colon \Spinc(Y) \to H^2(Y)
\]
is given by taking the first Chern class (or equivalently, Euler class)
of the oriented 2-plane field $v^{\perp}$, where $v$ is a nowhere
vanishing vector field representing $\s$. Then we have the
formula
\[
\s-\s' = 2(c_1(\s) - c_1(\s')).
\]
In particular, the homomorphism $c_1$ is injective if and only if
there is no 2-torsion in $H^2(Y)$.

Let $(\S,\alphas,\betas)$ be a Heegaard diagram, and consider the
tori $\T_\a$ and $\T_\b$ in $\sym^g(\S)$. We view~$D^2$ as the unit disk in~$\CC$.
Given intersection points
$\x$, $\y \in \T_\a \cap \T_\b$, a topological Whitney
disk from~$\x$ to~$\y$ is a continuous map $u \colon D^2 \to \sym^g(\S)$
such that
\begin{itemize}
\item $u(-i) = \x$ and $u(i) = \y$,
\item $u(\partial D^2 \cap \{\Re \le 0\}) \subset \T_\a $,
\item $u(\partial D^2 \cap \{\Re \ge 0\}) \subset \T_\b$.
\end{itemize}
A topological Whitney disk can be viewed as a path in $\Omega(\T_\a,\T_\b)$
connecting the constant paths~$\ell_\x$ and~$\ell_\y$.
We say that $u$ and $u'$ are \emph{homotopic} if they are homotopic
through topological Whitney disks. We denote by $\pi_2(\x,\y)$
the set of homotopy classes of Whitney disks from~$\x$ to~$\y$.
This is non-empty if and only if~$\ell_\x$ and~$\ell_\y$ lie in the
same component of $\Omega(\T_\a,\T_\b)$.

Given a based Heegaard diagram $\H = (\S, \alphas, \betas, z)$ of
$Y$, we can associate a $\Spinc$-structure $\s(\x)$ to every
intersection point $\x \in \T_\a \cap \T_\b$ as follows. Pick a
self-indexing Morse function~$f$ on~$Y$ with a unique minimum and
maximum, and a Riemannian metric~$g$, inducing~$\H$. Let~$v$ be the
gradient of~$f$ with respect to~$g$, and consider the flow of~$v$.
Then let~$N$ be a regular neighborhood of the union of the
flow-lines of~$v$ passing through the points of~$\x$ and through~$z$.
Then $v|_{Y \setminus N}$ extends to a nowhere vanishing vector field~$v'$
on~$Y$ since each component of~$N$ contains two singularities of~$v$
of opposite indices.  Then~$\s(\x)$ is defined to be the homology
class of~$v'$. Given intersection points~$\x$, $\y \in \T_\a \cap \T_\b$,
there is a topological Whitney disk connecting~$\x$ and~$\y$ if and only
if $\s(\x) = \s(\y)$. Note that~$\s(\x)$ depends on the choice of basepoint~$z$.

Now we explain the notion of admissibility.
One advantage of looking at
admissible diagrams is that the sums appearing in the definition of
the boundary map are finite; we do not need to use the completion of
Fukaya et al.\ and Novikov coefficients. Furthermore, it rules
out issues such as the one in the example in Figure~\ref{fig:admissibility},
as the diagram on the left hand side is not admissible.

Again, let $Y$ be a closed, connected,
oriented 3-manifold, and $\H = (\S,\alphas,\betas,z)$ a based Heegaard
diagram of~$Y$.
A \emph{domain} is a $\ZZ$-linear combination
of components of $\S \setminus (\alphas \cup \betas)$ (we call these \emph{regions}),
this can be viewed as a 2-chain on~$\S$. We denote by $D(\S,\alphas,\betas)$ the set of domains, this is
a free $\ZZ$-module. We define a map
\[
\cD \colon \pi_2(\x,\y) \to D(\S,\alphas,\betas)
\]
as follows: Given a homotopy class of topological Whitney disks
$\phi \in \pi_2(\T_\a,\T_\b)$, the coefficient of $\cD(\phi)$
at $p \in \S \setminus (\alphas \cup \betas)$ is
\[
n_p(\phi) = \# (\phi \cap V_p),
\]
where $V_p = \{p\} \times \sym^{g-1}(\S)$.
We say that a domain $\cD \in D(\S,\alphas,\betas)$
connects~$\x$ and~$\y$ if $\partial \cD  \cap \alphas$
is a 1-chain with boundary $\x -\y$, and $\partial \cD \cap \betas$
is a 1-chain with boundary $\y - \x$. Let $D(\x,\y)$ be the
set of such domains.
Essentially, the map~$\cD$ gives a bijection between $\pi_2(\x,\y)$
and $D(\x,\y)$. It is important to note that $\pi_2(\sym^g(\S)) = \ZZ$,
and if we view the generator as an element $\phi$ of $\pi_2(\x,\x)$, then
$\cD(\phi) = [\S]$; i.e., it has multiplicity one everywhere.

A \emph{periodic domain}~$\mathcal{P}$ in $\H$ is an element of $D(\S,\alphas,\betas)$
such that $n_z(\mathcal{P}) = 0$,
and $\partial \mathcal{P}$ is a linear combination of $\a$-circles and $\b$-circles
(as opposed to having $\a$-arcs or $\b$-arcs in $\partial
\mathcal{P}$). For every $\x \in \T_\a \cap \T_\b$ and $\cD \in D(\x,\x)$,
the domain $\cD - n_z(\cD) \cdot [\S]$ is a periodic domain.
In the symmetric product $\text{Sym}^g(\S)$, periodic
domains correspond to homology classes in $H_2(\text{Sym}^g(\S),
\mathbb{T}_\a \cup \mathbb{T}_\b)$ disjoint from $V_z$, or equivalently,
to elements of $\pi_1(\Omega(\T_\a,\T_\b))$ disjoint from~$V_z$.
Given a periodic domain~$\mathcal{P}$, we can cap off the boundary components
of $\mathcal{P}$ in $Y$ by the disks bounded by the $\alphas$- and $\betas$-curves
in the two handlebodies, and obtain a 2-cycle in~$Y$ whose
homology class we denote by $H(\mathcal{P}) \in H_2(Y)$.  We say that
a diagram is \emph{weakly admissible} if every non-zero periodic
domain has both positive and negative coefficients. This suffices
for the definition of $\widehat{HF}(Y)$ and $HF^+(Y)$. We can
always isotope the $\a$- and $\b$-curves in an arbitrary Heegaard
diagram to make it admissible.
Ozsv\'ath and Szab\'o~\cite{OSz1} showed that a diagram is weakly admissible if and only
if~$\S$ can be endowed with a volume form for which each periodic
domain has total signed area equal to zero.

To define the other flavors of Heegaard Floer homology, we need
\emph{strong admissibility}, which is in reference to a
$\Spinc$-structure $\s \in \Spinc(Y)$. A pointed Heegaard diagram
is called strongly admissible for the $\Spinc$ structure~$\s$
if for every non-trivial periodic domain~$\mathcal{P}$
with
\[
\langle c_1(\s), H(\mathcal{P}) \rangle = 2n \ge 0,
\]
the domain $\mathcal{P}$ has some coefficient bigger than~$n$.
If a diagram is strongly admissible for a single $\spinc$-structure, then
it is weakly admissible.
Furthermore, when $H^2(Y)$ is torsion; i.e., when $b_1(Y) = 0$, then
the notions of weak and strong admissibility coincide.
According to Lekili~\cite[Proposition~30]{Lekili} (the same proof also
covers the case $k = 0$), the diagram $(\S,\alphas,\betas,z)$
is strongly admissible for the $\Spinc$-structure~$\s$
if and only if the component of the path space $\Omega(\T_\a,\T_\b)$
corresponding to $\s$ is monotonic.
More precisely, if $\x \in \T_\a \cap \T_\b$ satisfies $\s(\x) = \s$,
then there exists an area form $\xi$ on $\S$ such that for an
induced symplectic form~$\omega_\xi$ on $\sym^g(\S)$ that makes~$\T_\a$ and~$\T_\b$ Lagrangian,
and the $\omega_\xi$-area and the index maps
\[
\pi_1(\Omega(\T_\a,\T_\b),\x) \to \RR
\]
are proportional.
Recall that monotonicity ensures that $\partial$ is well-defined, and
that the Lagrangian intersection Floer
homology is unobstructed; i.e., that $\partial^2 = 0$ (assuming the minimal Maslov number
is greater than two, or if the contributions of boundary bubbles to~$\partial^2$ is zero, which is the
case for~$HF$).
We can always isotope the $\a$- and $\b$-curves to make
the diagram strongly admissible for a given $\Spinc$-structure,
but as shown by the example of $S^1 \times S^2$, we cannot
necessarily make it strongly admissible for all of them simultaneously.

Let $(\S,\alphas,\betas,z)$ be strongly $\s$-admissible, and suppose that
$\s(\x_0) = \s$. Let $\Omega_{\x_0}(\T_\a,\T_\b)$ the component of $\ell_{\x_0}$.
Recall that to define Lagrangian intersection Floer homology, we need
to consider the cover $\wt{\Omega}_{\x_0}(\T_\a,\T_\b)$, and generators
are certain equivalence classes $[\ell_\x,w]$, where $w$ is a path of paths
from $\ell_{\x_0}$ to $\ell_\x$. Then we can view $w$ as a topological Whitney
disk from~$\x_0$ to $\x$. Given another such disk~$w'$, recall that $w \sim w'$
if they have the same Maslov index and symplectic area. Equivalently, if
$\phi = w \# \ol{w}'$ has Maslov index and symplectic area zero.
The domain $\cD(\phi)$ can be written as $\mathcal{P} + n_z(\phi) \cdot [\S]$.
By the strong admissibility condition, we can choose the symplectic form
on~$\S$ such that the signed area of $\mathcal{P}$ is zero, hence the symplectic
volume of the corresponding element of $\pi_2(\x_0,\x_0)$ is also zero.
By the result of Lekili, this also has Maslov index zero (this also follows from a
Maslov index formula of Ozsv\'ath and Szab\'o). The component of $\phi$ corresponding
to $n_z(\phi) \cdot [\S]$ is $n_z(\phi)$ times the generator of $\pi_2(\sym^g(\S))$.
This has Maslov index two. We conclude that $w \sim w'$ if and only if $n_z(w) = n_z(w')$,
and the map
\[
\pi_{\x_0} \colon \wt{\Omega}_{\x_0}(\T_\a,\T_\b) \to \Omega_{\x_0}(\T_\a,\T_\b)
\]
is a trivial $\ZZ$-covering. So strong $\s$-admissibility ensures that
the covering $\pi_{\x_0}$ of the component of the path-space
corresponding to~$\s$ is trivial, and the trivialization is given
by the choice of basepoint~$z$. In particular, we can identify a generator
$[\ell_\x,w]$ of the Floer chain complex with the pair
$[\x,n_z(w)] \in (\T_\a \cap \T_\b) \times \ZZ$. These are exactly
the generators of the chain complex $CF^\infty(\alphas,\betas,z)$,
as defined by Ozsv\'ath and Szab\'o. The group of deck transformations
of the covering $\pi_{\x_0}$ is $\ZZ$, Ozsv\'ath and Szab\'o
denote the generator by~$U$. To obtain $U[\ell_\x,w]$, we multiply
$w$ with the generator of $\pi_2(\sym^g(\S))$. This corresponds
to $U[\x,i] = [\x,i-1]$.

Given intersection points $\x$, $\y \in \T_\a \cap \T_\b$
with $\s(\x) = \s(\y)$, there is a homotopy class~$\phi$
of topological Whitney
disks from~$\x$ to~$\y$, and as in the case of Lagrangian
intersection Floer homology, we can consider the
moduli space $\M(\phi)$ of pseudo-holomorphic representatives
of $\phi$. This has an $\RR$-action, and we write
\[
\wM(\phi) = \M(\phi)/\,\RR.
\]
The expected dimension of $\M(\phi)$ is given by the Maslov
index $\mu(\phi)$. For the definition of the Heegaard Floer
differential, we are going to count rigid pseudo-holomorphic
discs; i.e., ones that have Maslov index one.
If $\mu(\phi) = 1$, then $\wM(\phi)$ is a finite collection
of points. We can either count the number of these points
modulo~$2$, or after an appropriate choice of orientations,
the points of $\wM(\phi)$ come with signs, and we can
take the algebraic number of them, $\# \wM(\phi)$.
As was mentioned before, all the moduli spaces are oriented if
the pair of Lagrangians $(\T_\a,\T_\b)$ is relatively spin.
In particular, it suffices to endow~$\T_\a$ and~$\T_\b$
with a $\text{Spin}$-structure.

Now we are ready to define the Heegaard Floer chain complex
$CF^\infty(\H, \s)$ for a strongly $\s$-admissible
pointed Heegaard diagram $\H = (\S,\alphas,\betas,z)$.  It is the free
$\ZZ$-module generated by pairs $[\x,i]$, where $\x \in \T_\a \cap
\T_\b$ is an intersection point, and $i \in \ZZ$ is an integer.
The grading is defined by
\[
\gr([\x,i],[\y,j]) = \mu(\phi) - 2n_z(\phi) + 2i -2j,
\]
where $\phi$ is a topological Whitney disk from $\x$ to $\y$,
and $n_z(\phi)$ is the algebraic intersection number of $\phi$
and $V_z$.
The boundary map $\partial^\infty$ is given by the formula
\[
\partial^\infty [\x,i] = \sum_{\y \in \T_\a \cap \T_\b} \sum_{\phi \in
  \pi_2(\x,\y) \atop \mu(\Phi) = 1} \# \wM(\phi) \cdot
[\y,i-n_z(\phi)].
\]
Note that if $\s(\y) \neq \s$, then the second sum is automatically
zero since there is not even a topological Whitney disk from~$\x$
to~$\y$. Hence, on the right, we only obtain generators $[\y,j]$
such that $\s(\y) = \s$. The fact that this sum is finite and we
do not have to resort to Novikov coefficients is again a consequence
of admissibility. Indeed, Ozsv\'ath and Szab\'o showed~\cite{OSz1}
that if the diagram is strongly $\s$-admissible, then,
given intersection points $\x$, $\y \in \T_\a \cap \T_\b$,
there are only finitely many homotopy classes
$\phi \in \pi_2(\x,\y)$ for which $\mu(\phi) = d$ and $\cD(\phi) \ge 0$
(i.e., all of its coefficients are non-negative).
Note that if $\phi$ has a pseudo-holomorphic representative, then $\cD(\phi) \ge 0$
by positivity of intersection with the hypersurfaces $V_p$ for
at least one point~$p$ in each component of $\S \setminus (\alphas \cup \betas)$
(we restrict ourselves to almost complex structures on $\sym^g(\S)$ that are holomorphic
around these~$V_p$).

If we choose a coherent system of orientations
for the moduli spaces, we get that
$\partial^\infty \circ \partial^\infty = 0$. The appropriate choice
of $\text{Spin}$-structures on $\T_\a$ and $\T_\b$ are obtained by
picking the non-fillable $\text{Spin}$-structure on each $\a_i$
and $\b_j$, and then taking the product of these. This is how we
canonically orient all the moduli spaces.

There is a chain map
\[
U \colon CF^\infty(\H,\s)
\to CF^\infty(\H,\s)
\]
defined by $U[\x,i] = [\x, i-1]$, which lowers degree by two.
As we mentioned above, this corresponds to a deck transformation
in the Floer theory.

We obtain the chain complexes $CF^-$, $CF^+$, and $\widehat{CF}$
from $CF^\infty$ as follows. Let $CF^-$ be the subcomplex of
$CF^\infty$ generated by pairs $[\x,i]$ such that $i < 0$.
Then $CF^+$ is the quotient complex $CF^\infty / CF^-$.
So there is a short exact sequence
\[
0 \longrightarrow CF^-(\H,\s) \longrightarrow
CF^\infty(\H,\s) \longrightarrow
CF^+(\H,\s) \longrightarrow 0.
\]
The endomorphism $U$ on $CF^\infty$ restricts to an endomorphism
$U^-$ on $CF^-$, and hence we also obtain an endomorphism $U^+$
on the quotient $CF^+$. We denote the kernel of $U^+$ by
$\widehat{CF}$. We can also define $\widehat{CF}(\S,\alphas,\betas,z,\s)$
by taking the free $\ZZ$-module generated by those intersection
points $\x \in \T_\a \cap \T_\b$ for which $\s(\x) = \s$,
and the boundary map $\widehat{\partial}$ is given by
\[
\widehat{\partial} \x = \sum_{y \in \T_\a \cap \T_\b}
\sum_{\phi \in \pi_2(\x,\y) \atop \mu(\phi) = 1, \, n_z(\phi) = 0} \#\wM(\phi) \cdot \y.
\]
This give rise to the short exact sequence
\[
0 \longrightarrow \widehat{CF}(\H,\s) \longrightarrow
CF^+(\H,\s) \stackrel{U^+}{\longrightarrow}
CF^+(\H,\s) \longrightarrow 0.
\]

By taking homology, we obtain the $\ZZ[U]$-modules
\[
HF^\infty(\H,\s) \text{, } HF^+(\H,\s) \text{, } HF^-(\H,\s) \text{, and } \HFh(\H,\s),
\]
where the $U$-action is trivial on $\HFh(\H,\s)$. These are related by long exact
sequences induced by the above short exact sequences of chain complexes.

\subsection{Invariance and naturality}

Let $Y$ be a closed, connected, oriented 3-man\-i\-fold, and $\s \in \spinc(Y)$ a $\spinc$-structure on~$Y$.
Ozsv\'ath and Szab\'o~\cite{OSz1} show that, given any two strongly $\s$-admissible pointed Heegaard diagrams
$\H = (\S,\alphas,\betas,z)$ and $\H' = (\S',\alphas',\betas',z')$ of~$Y$, one has
\[
HF^\circ(\H,\s) \cong HF^\circ(\H',\s),
\]
where $\circ$ is one of the four flavors hat, $+$, $-$, and $\infty$.
As mentioned before, by the Reidemeister-Singer theorem~\cite{Reidemeister, Singer},
the diagrams~$\H$ and~$\H'$ become \emph{diffeomorphic} after a sequence of Heegaard moves; i.e.,
isotopies of the $\a$- and $\b$-curves, $\a$- and $\b$-handleslides, stabilizations,
and destabilizations. In fact, Ozsv\'ath and Szab\'o show that there is also
a sequence of Heegaard moves that passes through strongly $\s$-admissible diagrams.
Then what remains to show is that $HF$ is invariant under changing the complex structure $\j$
on $\S$ and the $1$-parameter family $J_t$ of perturbations of $\sym^g(\j)$ on $\sym^g(\S)$,
and under any Heegaard move that preserves admissibility.

To assign a concrete group $HF^\circ(Y,p,\s)$ to the based $\spinc$ $3$-manifold $(Y,p,\s)$, one needs more.
The first steps toward naturality were made by Ozsv\'ath and Szab\'o~\cite{OSz10}, and
completely worked out by Thurston and the author~\cite{naturality} over $\ZZ_2$. What one needs
to construct is an isomorphism
\[
F_{\H,\H'}^\s \colon HF^\circ(\H,\s) \to HF^\circ(\H',\s)
\]
for any pair of admissible based Heegaard diagrams $\H$, $\H'$ of the based $3$-manifold $(Y,p)$
that satisfy the following two properties:
\begin{enumerate}
\item $F_{\H,\H}^\s = \text{Id}_{HF^\circ(\H,\s)}$ for any admissible diagram $\H$,
\item \label{it:nat} $F_{\H',\H''}^\s \circ F_{\H,\H'}^\s = F_{\H,\H''}^\s$ for any admissible diagrams $\H$, $\H'$, and $\H''$ of~$(Y,p)$.
\end{enumerate}
This is an instance of a \emph{transitive system of groups}, as defined by Eilenberg and Steenrod~\cite[Definition~6.1]{ES}.
We will call the maps~$F_{\H,\H'}$ \emph{canonical isomorphisms}. We would like to warn the reader of the
widespread practice of using the word ``canonical'' for any well-defined map, without checking property~\eqref{it:nat}.
Given such a transitive system, we obtain $HF^\circ(Y,p)$ by taking the product of all the groups
$HF^\circ(\H)$, where $\H$ is an admissible diagram of~$(Y,p)$ (note that these form a set as~$\S$ is a subset of ~$Y$),
and take elements $x$ in this product such that $F_{\H,\H'}^\s(x(\H)) = x(\H')$ for any pair $\H$, $\H'$.
For every admissible diagram~$\H$, we have an isomorphism
\[
P_\H^\s \colon HF^\circ(Y,p,\s) \to HF^\circ(\H,\s).
\]

One first constructs canonical isomorphisms for changing~$(\j,J_t)$ using continuation maps.
Next, suppose that $(\S,\alphas,\betas,z)$ and $(\S,\alphas,\betas',z)$ are diagrams of~$(Y,p)$
(in particular, $z = p$), and that
the triple diagram $(\S,\alphas,\betas,\betas',z)$ is admissible.
Note that $(\S,\betas,\betas')$ is a diagram
of $\#^g (S^1 \times S^2)$, and that the group $\HFh(\S,\betas,\betas',z,\s_0)$ is isomorphic
with $H_*(T^g; \ZZ_2)$. The ``fundamental class'' is denoted by~$\Theta_{\b,\b'}$.
Then we obtain a map
\[
\Psi^{\alphas}_{\betas \to \betas'} \colon HF^\circ(\S,\alphas,\betas,z) \to HF^\circ(\S,\alphas,\betas',z)
\]
by counting rigid pseudo-holomorphic triangles in $\sym^g(\S)$ with edges lying on $\T_\a$, $\T_\b$, and $\T_{\b'}$,
and one corner mapping to $\Theta_{\b,\b'}$.
Given admissible diagrams $\H = (\S,\alphas,\betas,z)$ and $\H' = (\S,\alphas',\betas',z)$ of~$(Y,p)$,
we construct $F_{\H,\H'}$ as follows.
First, suppose that the quadruple diagram $(\S,\alphas,\alphas',\betas,\betas')$ is admissible.
Then let
\[
\Psi^{\alphas \to \alphas'}_{\betas \to \betas'} = \Psi^{\alphas \to \alphas'}_{\betas'} \circ \Psi^{\alphas}_{\betas \to \betas'}.
\]
In the general case, we pick two sets of attaching curves $\ol{\alphas}$ and $\ol{\betas}$
such that the quadruple diagrams $(\S,\alphas,\ol{\alphas},\betas,\ol{\betas})$ and $(\S,\alphas',\ol{\alphas},\betas',\ol{\betas})$
are both admissible (this is always possible). Then let
\[
F_{\H,\H'} = \Psi^{\ol{\alphas} \to \alphas'}_{\ol{\betas} \to \betas'} \circ \Psi^{\alphas \to \ol{\alphas}}_{\betas \to \ol{\betas}}.
\]

Now suppose that $\H'$ is obtained from $\H$ by a stabilization. In this case, $\S' \setminus \S$
is a punctured torus~$T$, with a single $\a$-curve $\a_0$ and a single $\b$-curve $\b_0$ that intersect
in a unique point $\theta$. We obtain a map
\[
CF^\circ(\H) \to CF^\circ(\H')
\]
by mapping the generator $\x$ to $\x \times \theta$. This is a chain map when the complex structure
is chosen such that the connected sum neck along~$\partial T$ is very long. The map induced
on the homology is $F_{\H,\H'}$. Similarly, when $\H'$ is obtained from~$\H$ by
a destabilization, then we take $F_{\H,\H'} = F_{\H',\H}^{-1}$.

Finally, we define $F_{\H,\H'}$ when $\H$ and $\H'$ are related by a diffeomorphism.
In this case, $d$ induces a symplectomorphism between $\sym^g(\S)$ and $\sym^g(\S')$
that maps $\T_\a$ to $\T_{\a'}$ and $\T_\b$ to $\T_{\b'}$. If we use a complex structure $\j$ on $\S$
and $d_*(\j)$ on $\S'$, then $d$ tautologically induces a map between $HF^\circ(\H)$ and $HF^\circ(\H')$.

To obtain $F_{\H,\H'}$ for an arbitrary pair of diagrams of~$(Y,p)$, we take a sequence of diagrams
$\H_0, \dots, \H_n$ such that $\H_0 = \H$, $\H_n = \H'$, and $\H_i$ and $\H_{i+1}$
are related by changing the $\a$- and $\b$-curves, a stabilization or destabilization, or
a diffeomorphism that is isotopic to the identity in~$Y$ fixing~$p$. Then we let
\[
F_{\H,\H'} = F_{\H_{n-1},\H_n} \circ \dots \circ F_{\H_0,\H_1}.
\]
The main result of our paper with Thurston~\cite{naturality} is that this is independent of the
choice of sequence $\H_0,\dots,\H_n$. The idea of the proof is the following.
Consider the space of gradient vector fields~$v$ on~$Y$, these are vector fields that
arise as the gradient of a smooth function on~$Y$ with respect to some Riemannian metric.
Given a generic gradient vector field~$v$, there is a contractible space of overcomplete Heegaard diagrams
of~$Y$ compatible with~$v$ in the sense that~$\S$ is transverse to~$v$ and contains the index
zero and one critical points on its negative side, while the index two and three critical
points on its positive side. The $\a$-curves are obtained by taking the intersection of~$\S$
with the unstable manifolds of the index one critical points, while the $\b$-curves are obtained
by intersecting the stable manifolds of the index two critical points with~$\S$. We call the
diagram overcomplete because the number of $\a$- and $\b$-curves might exceed $g(\S)$.

Given a generic one-parameter family~$v_t$ of gradient vector fields, we can deform the associated
diagram smoothly as long as~$v_t$ is generic.
We then study what happens to the associated diagram as one passes a bifurcation value of~$t$.
These correspond to handleslides, a generalized form of (de)stabilization, and creation or
cancellation of 0-homologous $\a$- or $\b$-curves. By writing the generalized (de)stabilizations
as a sequence of regular Heegaard moves, and choosing suitable subdiagrams, this shows that
any pair of diagrams can be connected by a sequence of Heegaard moves and isotoping the
diagram in~$Y$.

To show that any two sequences of moves give the same map~$F_{\H,\H'}$, it suffices
to prove that if we compose the elementary maps along an arbitrary loop of diagrams,
we get the identity. Corresponding to this loop of moves, we can construct a loop of
gradient vector fields along~$\partial D^2$. Then we extend this to a generic
two-parameter vector field over~$D^2$. We subdivide the disk into small polygons, and
show that the composition around each small polygon is zero. For this end, one has to
understand the different types of codimension-two bifurcations that appear in such a
two-parameter family, and then translate these to loops of diagrams. There are many
cases which make the discussion rather complicated.
Along the codimension-one strata, one has to resolve the generalized (de)stabilizations,
pick suitable subdiagrams of the overcomplete diagrams,
and then write each small loop as a product of elementary loops.
Finally, we check that $HF$ has no monodromy for each such elementary loop.
Most of the elementary loops were already considered by Ozsv\'ath and
Szab\'o in~\cite{OSz10}, except for a new type of loop called a simple handleswap.

As mentioned earlier, $HF^\circ(Y,p)$ is independent of the choice of basepoint~$p$
when~$\circ$ is one of $+$, $-$, or $\infty$. However, $\HFh(Y,p)$ does depend
on~$p$, but the action of $\pi_1(Y,p)$ factors through $H_1(Y)/\text{Tors}$.

\subsection{Cobordism maps}

Given a $\spinc$-cobordism~$(W,\s)$ from $(Y_0,\s_0)$ to~$(Y_1,\s_1)$, Ozsv\'ath and Szab\'o associate to it
a homomorphism
\[
F^\circ_{W,\s} \colon HF^\circ(Y_0,\s_0) \to HF^\circ(Y_1,\s_1).
\]
This is defined in terms of a relative handle decomposition of~$W$ built on $Y_0 \times I$
with no 0- and 4-handles, and associating maps to each $i$-handle
attachment for $i \in \{1,2,3\}$. The 1-handle maps are similar to
the stabilization isomorphisms used in the proof of invariance: One
takes the connected sum of the diagram with a torus with a new $\a$- and a
new $\b$-curve that intersect in a single point~$\theta$, and map each intersection
point $\x \in \T_\a \cap \T_\b$ to $\x \times \{\theta\}$. This is a chain
map if the complex structure on $\S$ is chosen such that connected sum neck is very long.
The 3-handle map is the inverse of the one-handle map.

The key ingredient is the map associated to a 2-handle attachment.
Suppose that~$W$ is obtained by attaching a 2-handle to~$Y_0$ along a framed knot~$K$.
Take a diagram
\[
\left(\S, \alphas = \{\a_1,\dots,\a_g\}, \betas_K =  \{\b_1, \dots,\b_{g-1} \} \right)
\]
of $Y_0 \setminus N(K)$. Then $\b_g$ is chosen such that it represents the
meridian of~$K$, while $\b_g'$ is given by the surgery slope along~$K$, and we
require that $|\b_g \cap \b_g'| = 1$. We write
$\betas = \betas_K \cup \{\b_g\}$ and $\betas' = \betas_K' \cup \{\b_g'\}$,
where $\betas_K' = \{\b_1', \dots, \b_{g-1}'\}$ is a small isotopic copy of~$\betas_K$
such that $|\b_i \cap \b_i'| = 2$ for every $i \in \{1, \dots, g-1\}$.
Then $(\S,\alphas,\betas)$ is a diagram of $Y_0$, while $(\S,\alphas,\betas')$
is a diagram of~$Y_1$. Furthermore, $(\S,\betas, \betas')$ is a diagram of
$\#^{g-1} (S^1 \times S^2)$, and there is a distinguished ``top'' generator $\theta \in \T_\b \cap \T_{\b'}$.
The cobordism map is defined by counting pseudo-holomorphic triangles in
$\sym^g(\S)$ with edges mapping to $\T_\a$, $\T_\b$, and $\T_{\b'}$, respectively,
and such that one corner goes to~$\theta$, and whose homotopy class corresponds to~$\s$
in a suitable sense.

The difficult part is showing that the composition of
all the maps associated to the handle attachments is independent of
the choices made, including the choice of handle decomposition.
This composition is then denoted~$F_{W,\s}^\circ$.
For further details, we refer the reader to~\cite{OSz10, cob}.

\subsection{Computing {HF}}

What makes Heegaard Floer homology computable is the following observation.
Let $(\S,\alphas,\betas)$ be a Heegaard diagram, $\j$ a complex structure on $\S$, and let
\[
u \colon D^2 \to \sym^g(\S)
\]
be a pseudo-holomorphic disk with respect to $\sym^g(\j)$ connecting $\x$, $\y \in \T_\a \cap \T_\b$.
Then we can pull back the $g$-fold branched covering
\[
\S \times \sym^{g-1}(\S) \to \sym^g(\S)
\]
to obtain a $g$-fold branched covering
$p \colon S \to D^2$; i.e., the following diagram is commutative:
\[
\xymatrix{
  S \ar[r]^-{\overline{u}} \ar[d]^-p & \S \times \sym^{g-1}(\S) \ar[d] \ar[r]^-{\pi_1} & \S \\
  D^2  \ar[r]^-u & \sym^g(\S).}
\]
Let $\pi_1 \colon \S \times \sym^{g-1}(S) \to \S$ be the projection onto the first factor.
Then
\[
f = \pi_1 \circ \overline{u} \colon S \to \S
\]
is a holomorphic map.
The pair $(p,f)$ completely determines $u$. Indeed,
for $x \in D^2$, if $p^{-1}(x) = \{s_1, \dots, s_g\} \subset S$ (this is a multi-set,
with pre-images of branched points counted with multiplicity),
then
\[
u(x) = \{f(s_1), \dots, f(s_g)\} \in \sym^g(\S).
\]
In the opposite direction, given a $g$-fold branched cover $p \colon S \to D^2$
and a holomorphic map $f \colon S \to \S$ such that $f(p^{-1}(-i)) = \x$,
$f(p^{-1}(i)) = \y$, and the arcs in between $p^{-1}(-i)$ and $p^{-1}(i)$
alternatingly map to~$\alphas$ and~$\betas$, the above formula defines
a holomorphic representative of $\pi_2(\x,\y)$.

So finding the holomorphic representatives of a homotopy class $\phi \in \pi_2(\x,\y)$
is equivalent to finding $g$-fold branched coverings $p \colon S \to D^2$,
together with a holomorphic map $f \colon S \to \S$ such that $f_*[S]$ is the 2-chain $\cD(\phi)$,
and which satisfies the appropriate boundary conditions. Such maps~$f$ can be often found
using the Riemann mapping theorem, but in general there is no algorithm known for completely
determining their moduli space. Since the homology of the Heegaard Floer chain complex
(and more generally, its chain homotopy type) is independent of the choice of~$\j$ up to
isomorphism, in concrete computations it is often helpful to choose a degenerate
complex structure. If one knows some of the moduli spaces, together with the fact that
$\partial^2 = 0$, it is sometimes possible to work out~$\partial$ completely.

The above viewpoint was developed by Lipshitz~\cite{Lipshitz} into a self-contained
definition of Heegaard Floer homology that does not refer to Lagrangian intersection
Floer homology. He called this the ``cylindrical reformulation'' of Heegaard Floer homology.
Here, one studies pseudo-holomorphic curves in the 4-manifold $\S \times \RR \times I$
with suitable boundary conditions. Indeed, a holomorphic map
$v \colon S \to \S \times \RR \times I$ can be projected to the factors $\S$ and
$\RR \times I$ (which is conformally equivalent to $D^2 \setminus \{\pm i\}$)
to obtain a pair $(p,f)$, and vice versa.

An important result in his paper is
a formula for computing the Maslov index of an arbitrary domain $\cD$ connecting
intersection points $\x$ and $\y$. It consists of three terms. One is the Euler
measure $e(\cD)$ of $\cD$. To define this, assume that $\alphas$ and~$\betas$ meet at right
angles. If $S$ is a surface with $k$ acute right-angled corners and
$l$ obtuse right-angled corners, then
\[
e(S) = \chi(S) - k/4 + l/4.
\]
Every domain is a linear combination of such surfaces, and we extend~$e$ to domains linearly.
The other two terms are the averages $n_\x(\cD)$ and $n_\y(\cD)$ of the coefficients of~$\cD$ at the points
of~$\x$ and~$\y$, respectively (called point measures). So the Maslov index of~$\cD$ is given by
\[
\mu(\cD) = n_\x(\cD) + n_\y(\cD) + e(\cD).
\]

A homotopy class of Whitney disks~$\phi$ is called $\a$-injective if all of the multiplicities
of its domain~$\cD(\phi)$ are 0 or 1, if its interior
(i.e., the interior of the region with multiplicity 1) is disjoint from~$\alphas$, and its boundary contains
intervals in each~$\a_i$.
Using work of Oh~\cite{Oh2}, Ozsv\'ath and Szab\'o proved~\cite[Proposition~3.9]{OSz1} that
if a homotopy class of Whitney disks~$\phi$ is $\a$-injective, then we can perturb the $\a$-curves
such that the moduli space $\cM(\phi)$ of $\sym^g(\j)$-holomorphic disks is smoothly
cut out by its defining equation. So we can use the unperturbed complex
structure $\sym^g(\j)$ on $\sym^g(\S)$ in that case. This holds for example when the domain
of $\phi \in \pi_2(\x,\y)$ has coefficient one in an embedded rectangle or bigon, and is
zero elsewhere. In both cases, $\# \wM(\phi) = \pm 1$. This is the basis of the
algorithm of Sarkar and Wang~\cite{Suc} for computing~$\HFh$: One first performs isotopies
on the $\a$-curves in the Heegaard diagram $(\S,\alphas,\betas,z)$ until all the components
of $\S \setminus (\alphas \cup \betas)$ disjoint from~$z$ become bigons and rectangles.
Then they show that the domain of every rigid holomorphic disk with multiplicity zero at~$z$
is an embedded bigon or rectangle, and we saw above that each of these contributes~$\pm 1$
to the boundary map. Hence, the differential for the hat version becomes easy to compute,
without having to resort to complex analysis.

Manolescu, Ozsv\'ath, and Sarkar~\cite{MOS}
realized that one can use grid diagrams to algorithmically compute knot Floer homology.
A grid diagram is a multi-pointed Heegaard diagram on the torus, where every $\a$-curve
is a longitude and every $\b$-curve is a meridian, and these form a rectangular grid.
In each row and in each column there is exactly one~$z$ and one~$w$ basepoint.
This can be thought of as a sutured diagram for the knot complement, with several
sutures on each boundary torus. The Floer homology of this computes a stabilized
version of $\widehat{HFK}$, there is a factor of $\ZZ^2$ for each additional pair
of basepoints on each link component. Since every domain having multiplicity
zero at the basepoints is a rectangle, the boundary map
is completely combinatorial.

A more efficient algorithm for computing knot Floer homology is currently being
developed by Szab\'o, which is based on ideas coming from bordered Floer homology.
This can effectively compute the knot Floer homology of knots of 13 crossings, possibly
even more.

Ozsv\'ath and Szab\'o~\cite{OSz13} introduced an efficient algorithm
for computing Heegaard Floer homology for 3-manifolds obtained by plumbing
spheres along certain graphs. These include, for example, all Seifert fibred rational
homology spheres. Motivated by this algorithm, N\'emethi~\cite{lattice}
defined an invariant, called~\emph{lattice homology}, for any negative definite plumbed 3-manifold.
He conjectured that this agrees with~$HF^-$ for rational homology 3-spheres:

\begin{conj}
For any negative definite plumbing tree~$G$, the lattice homology $\mathbb{HF}^-(G)$
agrees with the Heegaard Floer homology $HF^-(Y_G)$ of the corresponding
3-manifold~$Y_G$.
\end{conj}

\section{Sutured Floer homology}

Sutured Floer homology, defined by
the author~\cite{sutured}, is an invariant of certain 3-manifolds with boundary.
It generalizes the hat version of Heegaard Floer homology,
basically by allowing multiple basepoints, and letting the number of $\a$- and $\b$-curves
differ from the genus of the Heegaard surface. We now review the necessary
definitions.

Sutured manifolds were introduced by Gabai~\cite{Gabai}, mainly to prove the
Property~R conjecture: If zero framed Dehn surgery on a knot~$K$ in~$S^3$
gives $S^1 \times S^2$, then $K$ is the unknot. A sutured manifold is
a pair $(M,\g)$, where~$M$ is a compact oriented 3-manifold with boundary,
and $\g \subset \partial M$, the sutures, is an oriented 1-manifold that divides the boundary
into two subsurfaces $R_+(\g)$ and $R_-(\g)$. The orientation of $R_+(\g)$
agrees with that of~$\partial M$, while~$R_-(\g)$ and $\partial M$ are oriented
oppositely.
Often, $\g$ is considered to be a closed regular neighborhood of the sutures,
while the sutures themselves are denoted by~$s(\g)$. In the latter case, one should think
of~$M$ as having corners along~$\partial \g$. We write $R(\g) = R_+(\g) \cup R_-(\g)$.
A sutured manifold $(M,\g)$ is \emph{taut} if $M$ is irreducible, and $R(\g)$ is incompressible
and Thurston norm minimizing in its homology class in $H_2(M,\g)$.

We say that $(M,\g)$ is a \emph{product sutured manifold} if there is a compact oriented
surface~$R$ such that $M = R \times I$ and $\g = \partial R \times I$. Recall that
after Conjecture~\ref{conj:free}, we defined the sutured manifolds
$Y(p)$ and $Y(L)$ for a based 3-manifold $(Y,p)$ and a link $L \subset Y$, respectively.
The sutured manifold $Y(p)$ is obtained from $Y$ by removing a ball about~$p$, and
putting a single suture on the resulting~$S^2$ boundary component. The sutured manifold
$Y(L)$ is obtained from $Y$ by removing a tubular neighborhood of~$L$, and the sutures
consist of two oppositely oriented meridians on each torus boundary component.

Gabai used sutured manifolds to study taut foliations on 3-manifolds. A foliation
of a sutured manifold $(M,\g)$ is a codimension-one, transversely oriented foliation that is transverse to~$\g$,
and such that each component of $R(\g)$ is a leaf. Such a foliation~$\cF$ is called \emph{taut} if
there is a properly embedded curve or arc in~$M$ that is transverse to the foliation, and
which intersects each leaf of~$\cF$ at least once. A deep theorem of Gabai~\cite{Gabai}
states that a sutured manifold~$(M,\g)$ carries a taut foliation if and only if it is taut.
The difficult direction is showing that if $(M,\g)$ is taut, then it carries a taut foliation.
The idea of the proof is the following. One can define a complexity for sutured manifolds,
which is minimal for products. Gabai showed that if~$(M,\g)$ is taut and not a product,
then there is always a properly embedded, oriented surface~$S$ such that if we cut~$M$ along~$S$
and add the negative side of~$S$ to~$R_-$ and the positive side to~$R_+$, then
the resulting sutured manifold $(M',\g')$ is of strictly smaller complexity. This operation
is called a sutured manifold decomposition, and is denoted by
\[
(M,\g) \rightsquigarrow^S (M',\g').
\]
Hence one obtains a sequence of decompositions
\[
(M,\g) \rightsquigarrow^{S_1} (M_1,\g_1) \rightsquigarrow^{S_2} \dots \rightsquigarrow^{S_n} (M_n,\g_n)
\]
resulting in a product sutured manifold $(M_n,\g_n)$;
he calls such a sequence a \emph{sutured manifold hierarchy}.
A product sutured manifold carries an obvious taut foliation, namely the product foliation.
Given a sutured manifold decomposition $(M,\g) \rightsquigarrow^S (M',\g')$ and a taut foliation~$\cF'$
on~$(M',\g')$, Gabai constructs a taut foliation~$\cF$ on~$(M,\g)$, assuming that the surface~$S$
is \emph{well-groomed}. This is a technical condition on~$S$, and the sutured manifold hierarchy
starting with $(M,\g)$ and ending in a product can be chosen such that each decomposing surface ~$S_i$
is well-groomed. Starting with the product foliation on $(M_n,\g_n)$, we end up with a taut foliation
on~$(M,\g)$.

A sutured manifold~$(M,\g)$ is called \emph{balanced} if $\chi(R_+(\g)) = \chi(R_-(\g))$,
the manifold~$M$ has no closed components, and each component of~$\partial M$ has at least
one suture. Note that the first condition automatically holds when~$(M,\g)$ is taut.
The author~\cite{sutured} defined sutured Floer homology, which assigns a
finitely generated Abelian group $SFH(M,\g)$ to every balanced sutured manifold~$(M,\g)$.
It splits along \emph{relative $\spinc$-structures}. We define these next. Let $v_0$ be a vector
field that points into~$M$ along~$R_-(\g)$, points out of~$M$ along~$R_+(\g)$, and which is
tangent to~$\partial M$ and points from $R_-(\g)$ to $R_+(\g)$ along the sutures.
A relative $\spinc$-structure on $(M,\g)$ is the homology class of a nowhere zero vector field~$v$ on~$M$ such that
$v|_{\partial M} = v_0$, where two such vector fields $v$ and $v'$ are said to be homologous
if they are homotopic through nowhere zero vector fields relative to~$\partial M$ in the complement of
a ball in the interior of~$M$. We denote the set of $\spinc$-structures by $\spinc(M,\g)$, this is
an affine space over~$H_1(M)$. As claimed above, we have a splitting
\[
SFH(M,\g) = \bigoplus_{\s \in \spinc(M,\g)} SFH(M,\g,\s).
\]

The group $SFH(M,\g)$ is constructed in the spirit of the hat version of Heegaard Floer homology
for closed 3-manifolds. We can represent every balanced sutured manifold~$(M,\g)$ by
a multi-pointed Heegaard diagram $(\ol{\S},\alphas,\betas,\z)$, where $\ol{\S}$ is a closed oriented surface,
$\alphas = \{\a_1,\dots,\a_d\}$ and $\betas = \{\b_1,\dots,\b_d\}$ are sets of pairwise disjoint oriented
simple closed curves on~$\ol{\S}$, and $\z = \{z_1,\dots, z_k\}$ is a set of points on $\ol{\S} \setminus (\alphas \cup \betas)$.
Given such a multi-pointed diagram, we can associate to it a sutured manifold as follows. Take
$\S = \ol{\S} \setminus N(\z)$, and let $M$ be the 3-manifold obtained by attaching 3-dimensional 2-handles
to $\S \times I$ along $\a_i \times \{0\}$ and $\b_j \times \{1\}$ for every $i$, $j \in \{1,\dots,d\}$.
Finally, we take $\g = \partial \S \times I$ and $s(\g) = \partial \S \times \{1/2\}$. The subsurface
$R_-(\g)$ of $\partial M$ is obtained from $\S \times \{0\}$ by compressing it along the curves~$\a_i \times \{0\}$.
Similarly, $R_+(\g)$ is obtained from $\S \times \{1\}$ by compressing it along the curves~$\b_j \times \{1\}$.
The condition that $\chi(R_-(\g)) = \chi(R_+(\g))$ is hence equivalent to having the same number of $\a$- and
$\b$-curves. There will be at least one suture on each component of $\partial M$ if and only if there is
at least one basepoint in each component of $\ol{\S} \setminus \alphas$ and $\ol{\S} \setminus \betas$.
For naturality purposes, we consider $\S$ to be a subsurface of~$M$ that divides~$M$ into two ``sutured compression
bodies.'' Each $\a$-curve and $\b$-curve bounds a disk in one of these two sutured compression bodies.

To define $SFH(M,\g)$, we consider the tori $\T_\a = \a_1 \times \dots \times \a_d$ and $\T_\b = \b_1 \times \dots \times \b_d$
inside the symmetric product $\sym^d(\ol{\S})$. So, in the definition of $\HFh(Y)$ for a closed
3-manifold~$Y$, it was just a coincidence that the exponent of the symmetric product was given by the
genus of the Heegaard surface, in general, it is the number of the $\a$- and $\b$-curves. The balanced
condition on~$(M,\g)$ (in particular, that $\chi(R_+(\g)) = \chi(R_-(\g))$)
is essential in ensuring that $\T_\a$ and $\T_\b$ are both half-dimensional submanifolds of $\sym^d(\ol{\S})$.
For a suitable symplectic form on $\sym^d(\ol{\S})$, the tori $\T_\a$ and $\T_\b$ will be Lagrangian.
Then $SFH(M,\g)$ is defined to be the Lagrangian intersection Floer homology of $\T_\a$ and $\T_\b$
in $\sym^d(\ol{\S})$, where the boundary map counts disks disjoint from the hypersurfaces
\[
V_{z_i} = \{z_i\} \times \sym^{d-1}(\ol{\S})
\]
for $i \in \{1,\dots,k\}$.
More concretely, given intersection points $\x$, $\y \in \T_\a \cap \T_\b$, the coefficient
of $\y$ in $\partial \x$ is given by counting the algebraic number of points in the moduli
space $\wt{\cM}(\phi)$, where $\phi$ is a homotopy class of pseudo-holomorphic Whitney disks
connecting $\x$ and $\y$ of Maslov-index one, and intersecting each $V_{z_i}$ algebraically zero times.
This way we obtain the chain complex $CF(\ol{\S},\alphas,\betas,\z)$, whose homology is $SFH(\ol{\S},\alphas,\betas,\z)$.
For different diagrams, in~\cite{naturality}, we construct canonical isomorphisms, and the limit
of the arising transitive system of groups is~$SFH(M,\g)$.

Just like in the case of $\HFh$, we need to assume that the diagram satisfies a weak
admissibility condition, namely, that every non-zero periodic domain has a positive and a negative
coefficient. This translates to a monotonicity condition in the language of Floer homology.
Since every component of $\ol{\S} \setminus \alphas$ and $\ol{\S} \setminus \betas$ contains a basepoint (as $(M,\g)$
is balanced), there are no non-zero periodic domains with boundary a linear combination of only $\a$-curves,
or only $\b$-curves. Hence, there are no homotopy classes of disks with boundary lying entirely on $\T_\a$ or on $\T_\b$,
and disjoint from the hypersurfaces~$V_{z_i}$. This ensures that in the sutured Floer chain complex $\partial^2 = 0$,
as in the boundary of a 1-parameter family of pseudo-holomorphic Whitney disks, we do not have bubbles with boundary
completely on one of the two Lagrangians, and only broken flow-lines appear.

Similarly to the closed case, we can assign a relative $\spinc$-structure to every intersection point
$\x \in \T_\a \cap \T_\b$ by first taking a self-indexing Morse-function $f \colon M \to \RR$ and a
Riemannian metric on~$M$ such that~$f$ has no index zero or three critical points,
$R_-(\g) = f^{-1}(0.5)$, $R_+(\g) = f^{-1}(2.5)$, $\S = f^{-1}(1.5)$,
and the gradient vector field~$v$ of~$f$ induces the diagram $(\S,\alphas,\betas)$ in the sense that the
unstable manifolds of index one critical points intersect $\S$ in the $\a$-curves, while the stable manifolds
of the index two critical points intersect~$\S$ in the $\b$-curves. Corresponding to~$\x$, there is
a multi-trajectory~$\g_\x$ of~$v$ connecting the index one and two critical points, and passing through the points
of~$\x$. If we modify~$v$ in a neighborhood of~$\g_\x$ to a nowhere zero vector field~$v'$, the homology class
of~$v'$ will be the $\spinc$-structure~$\s(\x)$. If there is a topological Whitney disk connecting
$\x$, $\y \in \T_\a \cap \T_\b$, then $\s(\x) = \s(\y)$, hence the intersection points lying in a given
$\spinc$-structure generate a subcomplex of $CF(\ol{\S},\alphas,\betas,\z)$.
Each summand $SFH(M,\g,\s)$ carries a relative $\ZZ_{d(c_1(\s))}$-grading, where
$d(c_1(\s))$ is the divisibility of the Chern class $c_1(\s) \in H^2(M)$.

They key property of sutured Floer homology is that it behaves particularly well under sutured manifold
decompositions. In particular, when performing a decomposition we get a subgroup.
Let $(S,\partial S) \subset (M,\partial M)$ be a properly embedded surface. We say that a $\spinc$-structure
$\s \in \spinc(M,\g)$ is \emph{outer} with respect to~$S$ if it can be represented by a nowhere zero vector field~$v$
on~$M$ such that $v_p \neq -(\nu_S)_p$ for every $p \in S$, where $\nu_S$ is the unit normal vector field
of~$S$ with respect to some Riemannian metric on~$M$.
Let $(M',\g')$ be the sutured manifold obtained by decomposing $(M,\g)$ along~$S$.
Then outer $\spinc$-structures are exactly the ones that arise by taking a relative $\spinc$-structure
on $(M',\g')$, and gluing $S_-$ to $S_+$.
Let $R$ be a compact oriented surface with no closed components. We say that a curve $C \subset R$
is \emph{boundary-coherent} if either $[C] \neq 0$ in $H_1(R)$, or if $[C] = 0$ and $C$ is oriented
as the boundary  of its interior (i.e., the component $R_1$ of $R \setminus C$ that is disjoint from~$\partial R$
and satisfies $\partial R_1 = C$).
We say that the decomposing surface~$S$ is \emph{nice} if it is open, and for every component~$V$ of $R(\g)$,
the set of closed components of $S \cap V$ consists of parallel oriented boundary-coherent simple closed
curves. In particular, every well-groomed decomposing surface, in the terminology of Gabai, is nice.
Now we are ready to state~\cite[Theorem~1.3]{decomposition}, the \emph{decomposition formula}.

\begin{theorem} \label{thm:decomp}
Let $(M,\g)$ be a balanced sutured manifold, and let $(M,\g) \rightsquigarrow^S (M',\g')$ be
a sutured manifold decomposition such that~$S$ is nice. Then
\[
SFH(M',\g') \cong \bigoplus_{\s \in O_S} SFH(M,\g,\s).
\]
In particular, $SFH(M',\g')$ is a direct summand of $SFH(M,\g)$.
\end{theorem}

\begin{proof}[Sketch of proof]
The idea is that if~$S$ is nice, then one can find a diagram $(\S,\alphas,\betas)$ of $(M,\g)$
that essentially contains~$S$ as a subsurface. More precisely, there is a subsurface~$P$ of~$\S$
such that $\partial P$ is a polygonal curve
on~$\S$ with corners being the points of~$P \cap \partial \S$ (recall that $\partial \S = s(\g)$).
Furthermore, we can label the edges of $\partial P$ alternatingly~$A$ and~$B$ (in particular, $\partial P = A \cup B$)
such that $A \cap \betas = \emptyset$, $B \cap \alphas = \emptyset$,
and we obtain~$S$ after smoothing the corners of
\[
(B \times [0,1/2]) \cup (P \times \{1/2\}) \cup (A \times [1/2,1]) \subset M.
\]
We call the tuple $(\S,\alphas,\betas,P)$ a \emph{surface diagram}.

Given a surface diagram, we can construct a diagram $(\S',\alphas',\betas')$ of $(M',\g')$ as follows.
To obtain~$\S'$, we take two copies~$P_A$ and~$P_B$ of~$P$, then glue $P_A$ to $\S \setminus P$ along~$A$,
and then we glue~$P_B$ to the resulting surface along~$B$. Loosely speaking, we have doubled~$P$. There is
a projection map $p \colon \S' \to \S$ that is two-to-one over~$P$ and is one-to-one over $\S \setminus P$.
Then we let $\alphas' = p^{-1}(\alphas) \setminus P_B$ and $\betas' = p^{-1}(\betas) \setminus P_A$; i.e., we
lift the $\a$-curves to~$P_A$ and the $\b$-curves to~$P_B$.

We call a generator  $\x \in \T_\a \cap \T_\b$ outer if $\x \cap P = \emptyset$. These are exactly the intersection
points with $\s(\x) \in O_S$. We denote the set of outer generators~$O_P$. The projection map~$p$
gives a bijection between $\T_{\a'} \cap \T_{\b'}$ and $O_P$. The difficult part of the proof is showing that
for some choice of $\alphas$ and $\betas$, this bijection is an isomorphism of chain complexes.
To achieve this, we elaborate on the Sarkar-Wang algorithm~\cite{Suc}, and wind the $\a$- and $\b$-curves until
each component of $\S \setminus (\alphas \cup \betas \cup A \cup B)$ disjoint from $\partial \S$
becomes either a bigon or a rectangle. Then one can show that the domain of every Maslov index one domain
in $(\S,\alphas,\betas)$ connecting two elements of~$O_P$ is an embedded bigon or rectangle that can be lifted
to a corresponding bigon or rectangle via the projection map~$p$. Hence, the map~$p$ establishes an isomorphism
between $CF(\S',\alphas',\betas')$ and the subcomplex of $CF(\S,\alphas,\betas)$ generated by~$O_P$.
\end{proof}

This theorem has several generalizations, which provide alternative proofs to the original
statement. One is a gluing formula for convex decomposition due to Honda, Kazez, and Mati\'c~\cite{TQFT}.
Given a convex decomposition $(M,\g) \rightsquigarrow (M',\g')$, they define a gluing map
$SFH(M',\g') \to SFH(M,\g)$, and show that this is an embedding for a sutured
manifold decomposition. Another generalization is due to Grigsby and Wehrli~\cite{GW},
who prove a decomposition formula for sutured-multi-diagrams via a neck-stretching argument.
A fourth proof is due to Zarev~\cite{Zarev}.

The decomposition formula has numerous nice implications. To be able to state these, we
first review the definition of the hat version of knot Floer homology in terms
of sutured Floer homology. Given a knot or link $K$ in a closed oriented 3-manifold~$Y$,
let $Y(K)$ denote the knot complement, together with two oppositely oriented meridional
sutures on each boundary torus. Then we let
\[
\widehat{HFK}(Y,K) = SFH(Y(K)).
\]
If we fix a Seifert surface~$S$ of $K$, the group $\widehat{HFK}(Y,K)$ carries
the so called Alexander grading, which
essentially comes from the grading by $\spinc(Y(K))$. The correspondence is
via evaluating a relative version of $c_1(\s)$ on $[S]$: if
$\langle c_1(\s) , [S] \rangle = 2i$, then $SFH(Y(K),\s)$ lies in Alexander grading~$i$.
The following is a special case of~\cite[Theorem~1.5]{decomposition}.

\begin{prop} \label{prop:HFK}
Let~$K$ be a null-homologous knot in a rational homology 3-sphere~$Y$, and let
$S$ be a Seifert surface of~$K$. Then
\[
SFH(Y(S)) \cong \widehat{HFK}(Y,K,g(S)).
\]
\end{prop}

\begin{proof}[Sketch of proof]
Recall that $Y(S)$ is the sutured manifold obtained by decomposing~$Y(K)$ along~$S$.
By the decomposition formula, all we need to do is identify the subset~$O_S$ of
$\spinc(Y(K))$ that survive the decomposition. It turns out that~$O_S$
consists of a single element,  characterized by $\langle c_1(\s), [S] \rangle = - 2g(S)$.
Hence
\[
SFH(Y(S)) \cong \widehat{HFK}(Y,K,-g(S)) \cong \widehat{HFK}(Y,K,g(S)),
\]
where the second isomorphism follows from a simple symmetry property of knot
Floer homology.
\end{proof}

This allows us to translate results on sutured Floer homology to knot Floer homology.
The following theorem says that sutured Floer homology detects tautness.

\begin{theorem} \label{thm:taut}
Let $(M,\g)$ be an irreducible balanced sutured manifold. Then $(M,\g)$ is taut
if and only if $SFH(M,\g) \neq 0$.
\end{theorem}

\begin{proof}[Sketch of proof]
Suppose that $(M,\g)$ is taut. Then, by the work of Gabai~\cite{Gabai}, there is a sutured manifold
hierarchy
\[
(M,\g) \rightsquigarrow^{S_1} (M_1,\g_1) \rightsquigarrow^{S_2} \dots \rightsquigarrow^{S_n} (M_n,\g_n)
\]
such that each surface~$S_i$ is well-groomed, and hence nice.
Here $(M_n,\g_n)$ is a product, so it has a diagram with no $\a$- and $\b$-curves, and
$SFH(M_n,\g_n) \cong \ZZ$. By the decomposition formula, $SFH(M_n,\g_n)$ is a direct summand of
$SFH(M,\g)$, which implies that $\text{rk}\, SFH(M,\g) \ge 1$. When $(M,\g)$ is not taut, then
one can construct a diagram where $\T_\a \cap \T_\b = \emptyset$.
\end{proof}

Using Proposition~\ref{prop:HFK}, the above theorem translates to the following result of
Ozsv\'ath and Szab\'o~\cite{OSz6}, which states that knot Floer homology detects the genus of a knot.

\begin{theorem}
Let $K$ be a knot in the rational homology 3-sphere~$Y$ with Seifert genus~$g(K)$. Then
\[
\widehat{HFK}(K,g(K)) \neq 0;
\]
moreover, $\widehat{HFK}(K,i) = 0$ for $i > g(K)$.
\end{theorem}

\begin{proof}
For every $i \ge g(K)$, we can choose a Seifert surface~$S$ for~$K$ such that $g(S) = i$.
Consider the decomposition $Y(K) \rightsquigarrow^S Y(S)$.
Proposition~\ref{prop:HFK} implies that $\widehat{HFK}(K,i) \cong SFH(Y(S))$.
The sutured manifold $Y(S)$ is irreducible, and it is taut if and only if $g(S) = g(K)$,
so the result follows from Theorem~\ref{thm:taut}.
\end{proof}

Note that the original proof of the above theorem also relied on Gabai's result on
sutured manifold hierarchies, plus a Theorem of Eliashberg and Thurston that
every taut foliation can be perturbed into a tight contact structure, and then showing
that the associated contact invariant in Heegaard Floer homology is non-zero via
Stein fillings and Lefshetz pencils. The proof presented here only relies on
Gabai's theorem and the decomposition formula for $SFH$.

We say that a decomposing surface~$S$ in $(M,\g)$ is \emph{horizontal} if it is open and incompressible,
$\partial S  = s(\g)$, $[S] = [R_+(\g)]$ in $H_2(M,\g)$, and $\chi(S) = \chi(R_+(\g))$.
The sutured manifold $(M,\g)$ is \emph{horizontally prime} if every horizontal surface is parallel
to either $R_+(\g)$ or $R_-(\g)$. It follows from the decomposition formula that decomposing along
a horizontal surface does not change~$SFH$.

A properly embedded annulus~$A$ in $(M,\g)$ is called a \emph{product annulus} if one component
of $\partial A$ lies in $R_-(\g)$, while the other component lies in $R_+(\g)$.
The sutured manifold $(M,\g)$ is said to be \emph{reduced} if every incompressible product
annulus~$A$ in $(M,\g)$ is isotopic to a component of~$\g$ such that $\partial A$ stays in
$R(\g)$ throughout. The author showed in~\cite{decomposition} that if we have a sutured manifold
decomposition $(M,\g) \rightsquigarrow^{A} (M',\g')$ such that~$A$ is a nice product annulus,
then $SFH(M,\g) \cong SFH(M',\g')$. However, note that the $\spinc$-gradings might be different on the
two groups.

Now we state the following important result from~\cite{decomposition}.

\begin{theorem} \label{thm:product}
Suppose that $(M,\g)$ is a taut balanced sutured manifold that is not a product. Then
\[
\text{rk}\, SFH(M,\g) \ge 2.
\]
\end{theorem}

\begin{proof}[Sketch of proof]
First, we note that if $H_2(M) \neq 0$, then $\chi(SFH(M,\g)) = 0$. Since $(M,\g)$ is taut,
$\text{rk}\, SFH(M,\g) \ge 1$, so in fact $\text{rk } SFH(M,\g) \ge 2$.
Hence, it suffices to consider the case when $H_2(M) = 0$.

Our goal is to construct two different taut decompositions
\[
(M,\g) \rightsquigarrow^{S_+} (M_+',\g_+')
\text{ and } (M,\g) \rightsquigarrow^{S_-} (M_-',\g_-')
\]
such that $O_{S_+} \cap O_{S_-} = \emptyset$.
If we can achieve this, then the decomposition formula implies that $SFH(M,\g)$ has a
subgroup isomorphic with $SFH(M_+',\g_+') \oplus SFH(M_-',\g_-')$, which is of rank at least two
since $(M_+',\g_+')$ and $(M_-',\g_-')$ are both taut.

Before we can carry out the above plan, we decompose $(M,\g)$ along horizontal surfaces and product
annuli until it becomes reduced and horizontally prime. This does not change $SFH(M,\g)$.
Then we take an arbitrary non-zero class $\a \in H_2(M, \partial M)$. There are nice taut
decomposing surfaces $S_+$ and $S_-$ such that $[S_+] = \a$ and $[S_-] = -\a$ in $H_2(M, \partial M)$.
(We can assume that $S_-$ and $S_+$ have no closed components since $H_2(M) = 0$.)
The rest of the argument uses cut-and-paste techniques to show that $O_{S_+} \cap O_{S_-} = \emptyset$,
assuming that $(M,\g)$ is reduced and horizontally prime.
\end{proof}

Recall that $SFH(M,\g) \cong \ZZ$ for a product sutured manifold $(M,\g)$.
So, if $(M,\g)$ is irreducible, then $SFH(M,\g) \cong \ZZ$ if and only if $(M,\g)$ is a product.
Indeed, if $(M,\g)$ is not taut, then $SFH(M,\g) = 0$, and if it is taut but
not a product, then $\text{rk}\, SFH(M,\g) \ge 2$.
Using Proposition~\ref{prop:HFK}, we can translate this to the following result, which
states that knot Floer homology detects fibred knots. The genus one case was proved by Ghiggini~\cite{Ghiggini},
and the general case was proved simultaneously
by Ni~\cite{fibred, corrigendum} and the author~\cite{decomposition, polytope}.

\begin{theorem}
Let $K$ be a null-homologous knot in the oriented 3-manifold~$Y$ such that $Y \setminus K$ is
irreducible, and let~$S$ be a Seifert surface for~$K$. Then
\[
\text{rk } \widehat{HFK}(Y,K,[S],g(S)) = 1
\]
if and only if~$K$ is fibred with fibre~$S$.
\end{theorem}

Given a balanced sutured manifold $(M,\g)$, let $S(M,\g) \subset \spinc(M,\g)$
be the support of $SFH(M,\g)$; i.e.,
\[
S(M,\g) = \{\, \s \in \spinc(M,\g) \colon SFH(M,\g,\s) \neq 0 \,\}.
\]
Choose an arbitrary affine isomorphism $i$ between $\spinc(M,\g)$ and $H_1(M;\ZZ)$,
and let $j \colon H_1(M;\ZZ) \to H_1(M,\RR)$ be the map induced by $\ZZ \hookrightarrow \RR$.
We denote by $P(M,\g)$ the convex hull of $j \circ i (S(M,\g))$ in $H_1(M;\RR)$, we call this
the \emph{sutured Floer homology polytope} of $(M,\g)$ (here, we have deviated slightly
from the conventions of~\cite{polytope} to simplify the discussion).
The main technical result of~\cite{polytope} is the following, which can be viewed as a
generalization of Theorem~\ref{thm:product}.

\begin{theorem}
Let $(M,\g)$ be a taut balanced sutured manifold that is reduced and horizontally prime,
and suppose that $H_2(M) = 0$. Then
\[
\dim P(M,\g) = \dim H_1(M; \RR).
\]
In particular,
\[
\text{rk } SFH(M,\g) \ge b_1(\partial M)/2 + 1.
\]
\end{theorem}

The above results can be used to study Seifert surfaces of knots and links.
There are several natural notions of equivalence between Seifert surfaces. We say that
the Seifert surfaces $R$ and $R'$ of a knot $K$ in $Y$ are \emph{strongly equivalent} if they
are isotopic in the knot complement $X(K) = Y \setminus N(K)$.
If $R$ is a Seifert surface of a knot $K$ and $R'$ is a Seifert surface of a knot $K'$,
then $S$ and $S'$ are \emph{weakly equivalent} if they are ambient isotopic in~$Y$.
The following is \cite[Theorem~2.3]{Seifert}.

\begin{theorem}
Let $K$ be a knot in~$S^3$ of genus~$g$, and let $n > 0$ be an integer. If
\[
\text{rk}\, \widehat{HFK}(K,g) < 2^{n+1},
\]
then $K$ has at most $n$ disjoint, pairwise strongly inequivalent minimal genus Seifert surfaces.
\end{theorem}

\begin{proof}
Suppose that $R_1, \dots, R_n$ are disjoint, strongly inequivalent minimal genus Seifert surfaces
for~$K$. Then $R_2, \dots, R_n$ are disjoint non-isotopic horizontal surfaces in
the sutured manifold $S^3(R_1)$ complementary to~$R_1$. If we decompose $S^3(R_1)$
along $R_2, \dots, R_n$, the rank of $SFH$ remains unchanged as each such surface
is null-homologous. We end up with the disjoint union of $n$ sutured manifolds, none
of which is a product as the~$R_i$ are non-isotopic. Hence, by Theorem~\ref{thm:product},
the~$SFH$ of each piece has rank at least two, and $\text{rk}\, SFH(S^3(R_1)) \ge 2^n$.
\end{proof}

As a special case, we obtain that if the leading coefficient of the Alexander polynomial
of an alternating knot has absolute value less than four, then it has a unique minimal genus Seifert
surface up to strong equivalence. This was not known before. An elementary proof of this
fact was later given by Banks~\cite{Banks}.

$SFH$ is also useful for distinguishing Seifert surfaces up to weak equivalence.
Given Seifert surfaces~$R$ and~$R'$ in~$S^3$ such that $\partial R$ and $\partial R'$
are ambient isotopic, Proposition~\ref{prop:HFK} might be discouraging as it implies that
\[
SFH(S^3(R)) \cong SFH(S^3(R')).
\]
However, recall that these groups are graded by $\spinc$-structures, which are affine
spaces over $H_1(S^3 \setminus R)$ and $H_1(S^3 \setminus R')$, respectively.
Often, even these graded groups are different. However, in many of these cases,
if we also consider the Seifert forms on $H_1(R) \cong H_1(S^3 \setminus R)$
and on $H_1(R') \cong H_1(S^3 \setminus R')$ (where the isomorphisms are given
by Alexander duality), we can distinguish~$R$ and~$R'$. Note that the Seifert
form encodes the way $R_+$ and $R_-$ are glued together in~$S^3(R)$. This
idea was presented by Hedden, Sarkar, and the author~\cite{HJS}.
The first example of Seifert surfaces $R$ and $R'$ where the graded groups themselves are different
was given by Altman~\cite{Irida}.

Sutured Floer homology currently has extensions in two different directions. One is bordered sutured Floer homology,
due to Zarev~\cite{Zarev}. In his theory, part of the boundary of the manifold is sutured,
part of it is bordered. He obtains various nice gluing results in this context.
The second is a generalization of the minus version of Heegaard Floer homology to sutured manifolds,
due to Alishahi and Eftekhary~\cite{AE}. The chain complex they define is over an algebra depending
on the sutures, and is well-defined up to chain homotopy equivalence. The relations in the algebra
correspond to the disks with boundary entirely in $\T_\a$ or $\T_\b$, and which obstruct~$\partial^2 = 0$.

Kronheimer and Mrowka~\cite{propertyP} extended instanton and monopole Floer homology
to balanced sutured manifolds. Applied to the sutured manifold complementary
to a link, they obtained new link invariants in the instanton and monopole settings.
This led to a new, considerably simpler proof of the Property~P conjecture,
and a proof of the fact that Khovanov homology detects the unknot, see~\cite{unknot}.
We now outline the proof of the latter result. Kronheimer and Mrowka~\cite{sing-inst}
defined another invariant of links using singular instantons, and showed that
for knots it agrees with the sutured instanton knot invariant. They
exhibited that there is a spectral sequence starting from reduced Khovanov homology
and converging to the singular instanton knot invariant.
The key step in the construction of the spectral sequence is showing that
the singular instanton knot invariant satisfies an unoriented skein exact triangle.
Finally, they proved an
analogue of the decomposition formula for the sutured instanton invariant,
which, using the arguments outlined earlier, implies that the sutured instanton knot invariant
detects the Seifert genus, and hence in particular it detects the unknot.
Another consequence of the above discussion is that if one could show that knot Floer homology is
isomorphic to the sutured instanton knot invariant, then one would
get a positive answer to Conjecture~\ref{conj:ss} for knots; i.e., that there is a spectral sequence
from reduced Khovanov homology to knot Floer homology.

It is worth mentioning that bordered Floer homology, just like $SFH$, is also defined using Heegaard surfaces with boundary.
But whereas in $SFH$ all the $\a$- and $\b$-curves lie in the interior of~$\S$, in the bordered
theory one might also have $\a$- or $\b$- arcs with ends on~$\partial \S$.
Then they consider Lipshitz's cylindrical reformulation~\cite{Lipshitz}, and count curves
in the 4-manifold $\S \times I \times \RR$. The complex structure on~$\S$ is chosen such that
the boundary becomes a puncture.
The bordered algebra associated with the boundary of the manifold encodes how holomorphic curves
limit to~$\partial \S$. The theory lends itself to nice gluing formulas, obtained by cutting an
ordinary Heegaard diagram into two pieces along a curve.

\bibliographystyle{amsplain} \bibliography{topology}

\end{document}